\documentclass[a4paper,oneside,12pt]{amsart}
\usepackage{mathrsfs}
\usepackage{amssymb}
\usepackage{amsmath,amstext,amsthm,amscd}
\usepackage{dsfont}
\usepackage{mathrsfs}

\usepackage[a4paper,width=15.5cm,top=2cm,bottom=2cm,footskip=1cm]{geometry}

\makeatletter
      \def\@setcopyright{}
      \def\serieslogo@{}
      \makeatother

\newcommand{\Complex}{\mathbb C}
\newcommand{\Real}{\mathbb R}

\newcommand{\ddbar}{\overline\partial}
\newcommand{\pr}{\partial}
\newcommand{\ol}{\overline}
\newcommand{\Td}{\widetilde}
\newcommand{\norm}[1]{\left\Vert#1\right\Vert}
\newcommand{\abs}[1]{\left\vert#1\right\vert}
\newcommand{\rabs}[1]{\left.#1\right\vert}
\newcommand{\set}[1]{\left\{#1\right\}}
\newcommand{\To}{\rightarrow}

\newcommand{\pit}{\mathit\Pi}
\newcommand{\cali}[1]{\mathscr{#1}}

\theoremstyle{plain}
\newtheorem{thm}{Theorem}[section]
\newtheorem{cor}[thm]{Corollary}
\newtheorem{lem}[thm]{Lemma}
\newtheorem{prop}[thm]{Proposition}
\newtheorem{ass}[thm]{Assumption}
\theoremstyle{definition}

\theoremstyle{remark}
\newtheorem{rem}[thm]{Remark}

\numberwithin{equation}{section}

\begin{document}
\title[]{On the coefficients of the asymptotic expansion of the kernel of Berezin-Toeplitz quantization}
\author[]{Chin-Yu Hsiao}
\address{Universit{\"a}t zu K{\"o}ln,  Mathematisches Institut,
    Weyertal 86-90,   50931 K{\"o}ln, Germany}
\thanks{The author is supported by the DFG funded project MA 2469/2-1}
\email{chsiao@math.uni-koeln.de}

\begin{abstract} 
We give new methods for computing the coefficients of the asymptotic expansions of the kernel of Berezin-Toeplitz quantization
obtained recently by Ma-Marinescu, and of the composition of two Berezin-Toeplitz quantizations. 
Our main tool is the stationary phase formula of Melin-Sj\"{o}strand.
\end{abstract}    

\maketitle \tableofcontents  

\section{Introduction and statement of the main results}  

Let $L^k$ be the $k$-th tensor power of a positive holomorphic line bundle $L$ over a compact complex manifold $X$. 
Let $\cali{H}^0(X,L^k)$ be the space of global holomorphic sections of $L^k$ and let
$\pit^{(k)}$ denote the orthogonal projector on $\cali{H}^0(X,L^k)$ in the $L^2$ space. Let $f\in C^\infty(X)$. 
Berezin-Toeplitz quantization with symbol $f$ is the operator $T^{(k)}_f=\pit^{(k)}\circ f\circ\pit^{(k)}$. 
The study of the $k$ large behaviour of $T^{(k)}_f$ is important in the geometric quantization theory . Ma and Marinescu~\cite{MM08b} obtained a full off-diagonal asymptotic expansion in $k$ of the kernel of $T^{(k)}_f$. 
They also calculated in~\cite{MM10} the first three coefficients of the expansion on the diagonal when $X$ is polarized (see \eqref{s1-e17}) by using kernel calculations on $\Complex^n$ and the analytic localization technique of Bismut-Lebeau~\cite{MM07}.
The coefficients of the expansion turned out to be deeply related to various problem in complex geometry (see e.g. Fine~\cite{Fine10a},~\cite{Fine10b} ). 

Since microlocal analysis is one of the main tools of quantization, it is important to know how to use microlocal analysis to calculate the coefficients of the expansion. This will be done in the present paper. There are two ingredients of our approach: the phase function version of the asymptotic expansion of the kernel of Berezin-Toeplitz quantization and the method of stationary phase. We will calculate the first three coefficients of the expansion by using this method. We do not assume that $X$ is polarized. Even through the inspiration for the calculation from microlocal analysis, the arguments in this paper are elementary and simple. 

\subsection{Some standard notations}

We shall use the following notations: $\Real$ is the set of real numbers, $\mathbb N=\set{1,2,\ldots}$, $\mathbb N_0=\mathbb N\bigcup\set{0}$. An element $\alpha=(\alpha_1,\ldots,\alpha_n)$ of $\mathbb N_0^n$ will be called a multiindex and the length of $\alpha$ is: $\abs{\alpha}=\alpha_1+\cdots+\alpha_n$. We write $x^\alpha=x_1^{\alpha_1}\cdots x^{\alpha_n}_n$, 
$x=(x_1,\ldots,x_n)$,
$\pr^\alpha_x=\pr^{\alpha_1}_{x_1}\cdots\pr^{\alpha_n}_{x_n}$, $\pr_{x_j}=\frac{\pr}{\pr x_j}$, $D^\alpha_x=D^{\alpha_1}_{x_1}\cdots D^{\alpha_n}_{x_n}$, $D_x=\frac{1}{i}\pr_x$, $D_{x_j}=\frac{1}{i}\pr_{x_j}$. 

Let $M$ be a $C^\infty$ paracompact manifold. 
We let $TM$ and $T^*M$ denote the tangent bundle of $M$ and the cotangent bundle of $M$ respectively.
The complexified tangent bundle of $M$ and the complexified cotangent bundle of $M$ will be denoted by $\Complex TM$
and $\Complex T^*M$ respectively. We write $\langle\,\cdot\,,\cdot\,\rangle$ to denote the pointwise duality 
between $TM$ and $T^*M$.
We extend $\langle\,\cdot\,,\cdot\,\rangle$ bilinearly to $\Complex TM\times\Complex T^*M$.
Let $E$ be a $C^\infty$ vector bundle over $M$. The fiber of $E$ at $x\in M$ will be denoted by $E_x$.
Let $F$ be another $C^\infty$ vector bundle over $M$. We write 
$E\boxtimes F$ to denote the vector bundle over $M\times M$ with fiber over $(x, y)\in M\times M$ 
consisting of the linear maps from $E_x$ to $F_y$. We write ${\rm End\,}(E)$ to denote $E\boxtimes E$. 

\subsection{Set up and Terminology} 

Let $X$ be a compact complex manifold of dimension $n$ with a smooth
Hermitian metric $\langle\,\cdot\,,\cdot\,\rangle$ on the holomorphic tangent bundle $T^{1,0}X$.
Let $T^{0,1}X$ be the anti-holomorphic tangent bundle of $X$. We extend the Hermitian metric $\langle\,\cdot\,,\cdot\,\rangle$ to $\Complex TX$ in a natural way by requiring $T^{1,0}X$ 
to be orthogonal to $T^{0,1}X$ and satisfy $\ol{\langle\,u\,,v\,\rangle}=\langle\,\ol u\,,\ol v\,\rangle$, $u, v\in T^{0,1}X$. For $p, q\in\mathbb N_0$, let $\Lambda^{p,q}T^*X$ be the bundle of $(p,q)$ forms of $X$. The Hermitian metric $\langle\,\cdot\,,\cdot\,\rangle$ on $\Complex TX$ induces a Hermitian metric on
$\Lambda^{p, q}T^*X\otimes\Lambda^{r,s}T^*X$, $p,q,r,s\in\mathbb N_0$, also denoted by $\langle\,\cdot\,,\cdot\,\rangle$. Let
$D\subset X$ be an open set. If $E$ is a vector bundle over $D$, then we let $C^\infty(D, E)$
denote the space of smooth sections of $E$ over $D$. Let $C^\infty_0(D, E)$ be the subspace of
$C^\infty(D, E)$ whose elements have compact support in $D$. 

In the sequel we will denote by $\langle\,\cdot\,,\cdot\,\rangle$ both scalar products as well as the duality bracket between vector fields and forms. 

Let $(L,h^L)$ be a holomorphic line bundle over $X$, where
the Hermitian fiber metric on $L$ is denoted by $h^L$. In this work, we assume that $h^L$ is smooth. 
Let $\phi$ denote the local weights of the Hermitian metric. More precisely, if
$s(z)$ is a local trivializing section of $L$ on an open subset $D\subset X$, then the pointwise
norm of $s$ is
\begin{equation} \label{s1-e2}
\abs{s(x)}^2=\abs{s(x)}^2_{h^L}=e^{-2\phi(x)},\quad\phi\in C^\infty(D, \Real).
\end{equation}
Let $R^L$ be the canonical curvature two form induced by $h^L$. In terms 
of the local weight $\phi$, we have $R^L=2\pr\ddbar\phi$.

We will identify the curvature two form $R^L$ with the Hermitian matrix 
\[\dot{R}^L\in C^\infty(X, {\rm End\,}(T^{1,0}X))\] 
such that for $U, V\in T^{1,0}_xX$, $x\in X$, we have 
\begin{equation} \label{s1-e3}
\langle\,\dot{R}^L(x)U\,,V\,\rangle=\langle\,R^L(x)\,,U\wedge\ol V\,\rangle.
\end{equation}
In this work, we assume that 

\begin{ass} \label{s1-a1} 
$\dot{R}^L$ is positive at each point of $X$, that is, $L$ is a positive holomorphic line bundle over $X$.
\end{ass} 

We introduce now the geometric objects used in Theorem~\ref{s1-tmain}, Theorem~\ref{s1-tmaincomp} below. Put 
\begin{equation} \label{s1-e6} 
\omega=:\frac{\sqrt{-1}}{2\pi}R^L. 
\end{equation}
The real two form $\omega$ induces a Hermitian metric $\langle\,\cdot\,,\cdot\,\rangle_{\omega}$ on $\Complex TX$. 
The Hermitian metric $\langle\,\cdot\,,\cdot\,\rangle_{\omega}$ on $\Complex TX$ 
induces a Hermitian metric on
$\Lambda^{p, q}T^*X\otimes\Lambda^{r, s}T^*X$, $p, q, r, s\in\mathbb N_0$, also denoted by $\langle\,\cdot\,,\cdot\,\rangle_{\omega}$. 
For $u\in\Lambda^{p,q}T^*X$, we denote $\abs{u}^2_{\omega}=:\langle\,u,u\,\rangle_{\omega}$.  

Let $\Theta$ be the real two form induced by $\langle\,\cdot\,,\cdot\,\rangle$. 
In local holomorphic coordinates $z=(z_1,\ldots,z_n)$, put
\begin{equation} \label{s1-e7} 
\begin{split}
&\omega=\sqrt{-1}\sum^n_{j,k=1}\omega_{j,k}dz_j\wedge d\ol z_k,\\
&\Theta=\sqrt{-1}\sum^n_{j,k=1}\Theta_{j,k}dz_j\wedge d\ol z_k. 
\end{split}
\end{equation}
We notice that $\Theta_{j,k}=\langle\,\frac{\pr}{\pr z_j}\,,\frac{\pr}{\pr z_k}\,\rangle$, $\omega_{j,k}=\langle\,\frac{\pr}{\pr z_j}\,,\frac{\pr}{\pr z_k}\,\rangle_{\omega}$, $j,k=1,\ldots,n$.
Put 
\begin{equation} \label{s1-e8}
h=\left(h_{j,k}\right)^n_{j,k=1},\ \ h_{j,k}=\omega_{k,j},\ \ j, k=1,\ldots,n,
\end{equation}
and $h^{-1}=\left(h^{j,k}\right)^n_{j,k=1}$, $h^{-1}$ is the inverse matrix of $h$. The complex Laplacian with respect to $\omega$ is given by 
\begin{equation} \label{s1-e9} 
\triangle_{\omega}=(-2)\sum^n_{j,k=1}h^{j,k}\frac{\pr^2}{\pr z_j\pr\ol z_k}.
\end{equation} 
We notice that $h^{j,k}=\langle\,dz_j\,,dz_k\,\rangle_{\omega}$, $j, k=1,\ldots,n$. Put 
\begin{equation} \label{s1-e10} 
\begin{split}
&V_\omega=:\det\left(\omega_{j,k}\right)^n_{j,k=1},\\
&V_\Theta=:\det\left(\Theta_{j,k}\right)^n_{j,k=1}
\end{split}
\end{equation}
and set  
\begin{equation} \label{s1-e11} 
\begin{split}
&r=\triangle_{\omega}\log V_\omega,\\
&\hat r=\triangle_{\omega}\log V_\Theta.
\end{split}
\end{equation} 
$r$ is called the scalar curvature with respect to $\omega$. Let $R^{\det}_\Theta$ be the curvature of the determinant line bundle of $T^{1,0}X$ with respect to the real two form $\Theta$. We recall that 
\begin{equation} \label{s1-e12}
R^{\rm det\,}_\Theta=-\ddbar\pr\log V_\Theta.
\end{equation} 

Let $h$ be as in \eqref{s1-e8}. Put 
$\theta=h^{-1}\pr h=\left(\theta_{j,k}\right)^n_{j,k=1}$, $\theta_{j,k}\in\Lambda^{1,0}T^*X$, $j,k=1,\ldots,n$.
$\theta$ is the Chern connection 
matrix with respect to $\omega$. The Chern curvature with respect to $\omega$ is given by
\begin{equation} \label{s1-e13}
\begin{split}
&R^{TX}_{\omega}=\ddbar\theta=\left(\ddbar\theta_{j,k}\right)^n_{j,k=1}=\left(\mathcal{R}_{j,k}\right)^n_{j,k=1}
\in C^\infty(X,\Lambda^{1,1}T^*X\otimes{\rm End\,}(T^{1,0}X)),\\
&R^{TX}_{\omega}(\ol U, V)\in {\rm End\,}(T^{1,0}X),\ \ \forall U, V\in T^{1,0}X,\\
&R^{TX}_\omega(\ol U,V)\xi=\sum^n_{j,k=1}\langle\,\mathcal{R}_{j,k}\,,\ol U\wedge V\,\rangle\xi_k\frac{\pr}{\pr z_j},\ \ \xi=\sum^n_{j=1}\xi_j\frac{\pr}{\pr z_j},\ \ U, V\in T^{1,0}X.
\end{split}
\end{equation} 
Set 
\begin{equation} \label{s1-e14} 
\abs{R^{TX}_{\omega}}^2_{\omega}=:\sum^n_{j,k,s,t=1}\abs{\langle\,R^{TX}_{\omega}(\ol e_j,e_k)e_s\,,e_t\,\rangle_{\omega}}^2,
\end{equation} 
where $e_1,\ldots,e_n$ is an orthonormal frame for $T^{1,0}X$ with respect to $\langle\,\cdot\,,\cdot\,\rangle_{\omega}$. 
It is straightforward to see that the definition of $\abs{R^{TX}_{\omega}}^2_{\omega}$ is independent of the choices of orthonormal frames. Thus, $\abs{R^{TX}_{\omega}}^2_{\omega}$ is globally defined. The Ricci curvature with respect to $\omega$ is given by 
\begin{equation} \label{s1-e15}
{\rm Ric\,}_{\omega}=:-\sum^n_{j=1}\langle\,R^{TX}_\omega(\cdot,e_j)\cdot\,,e_j\,\rangle_\omega,
\end{equation}
where $e_1,\ldots,e_n$ is an orthonormal frame for $T^{1,0}X$ with respect to $\langle\,\cdot\,,\cdot\,\rangle_{\omega}$. That is, 
\[\langle\,{\rm Ric\,}_{\omega}\,, U\wedge V\,\rangle=-\sum^n_{j=1}\langle\,R^{TX}_\omega(U,e_j)V\,,e_j\,\rangle_\omega,\ \ 
U, V\in\Complex TX.\]
${\rm Ric\,}_{\omega}$ is a global $(1,1)$ form.

Let 
\begin{equation} \label{s1-e15-I}
D^{0,1}:C^\infty(X,\Lambda^{0,1}T^*X)\To C^\infty(X,\Lambda^{0,1}T^*X\otimes\Lambda^{0,1}T^*X)
\end{equation}
be the $(0,1)$ component of the Chern connection on $\Lambda^{0,1}T^*X$ induced by $\langle\,\cdot\,,\cdot\,\rangle_\omega$. That is, in local coordinates $z=(z_1,\ldots,z_n)$, put 
\[A=\left(a_{j,k}\right)^n_{j,k=1},\ \ a_{j,k}=\langle\,d\ol z_k\,,d\ol z_j\,\rangle_\omega,\ \ j,k=1,\ldots,n,\] 
and set 
\begin{equation} \label{s1-e15-Ibis}
\mathcal{A}=A^{-1}\pr A=\left(\alpha_{j,k}\right)^n_{j,k=1},\ \ \alpha_{j,k}\in\Lambda^{0,1}T^*X,\ \ j,k=1,\ldots,n.
\end{equation}
Then, for $u=\sum^n_{j=1}u_jd\ol z_j\in C^\infty(X,\Lambda^{0,1}T^*X)$, we have 
\[D^{0,1}u=\sum^n_{j=1}\ddbar u_j\otimes d\ol z_j+\sum^n_{j,k=1}u_j\alpha_{k,j}\otimes d\ol z_k\in C^\infty(X,\Lambda^{0,1}T^*X\otimes\Lambda^{0,1}T^*X).\] 

Similarly, let 
\begin{equation} \label{s1-e15-II}
D^{1,0}:C^\infty(X,\Lambda^{1,0}T^*X)\To C^\infty(X,\Lambda^{1,0}T^*X\otimes\Lambda^{1,0}T^*X)
\end{equation}
be the $(1,0)$ component of the Chern connection on $\Lambda^{1,0}T^*X$ induced by $\langle\,\cdot\,,\cdot\,\rangle_\omega$. 

\subsection{The main results} 

In order to state our results precisely, we first review briefly the asymptotic expansion of the 
kernel of Berezin-Toeplitz quantization.
Let $L^k$, $k>0$, be the $k$-th tensor power of the line bundle $L$.
The Hermitian fiber metric on $L$ induces a Hermitian fiber metric
on $L^k$ that we shall denote by $h^{L^k}$. If $s$ is a local
trivializing section of $L$ then $s^k$ is a local trivializing
section of $L^k$. For $f\in C^\infty(X, L^k)$, we denote the
pointwise norm $\abs{f(x)}^2=:\abs{f(x)}^2_{h^{L^k}}$. We denote by 
\[dv_X=dv_X(x)=\frac{\Theta^n}{n!}\]
the volume form on $X$ induced by the fixed Hermitian metric $\langle\,\cdot\,,\cdot\,\rangle$ on $\Complex TX$. Then we get natural inner products $(\ |\ )_k$, $(\ |\ )$ on $C^\infty(X, L^k)$ and
$C^\infty(X)$ respectively. More precisely, let $s$ be a local
trivializing section of $L$ on an open set $D\subset X$, $\abs{s}^2=e^{-2\phi}$, then for $u=s^k\Td u, v=s^k\Td v\in C^\infty_0(D,L^k)$, we have
\begin{equation} \label{e-inner} 
(u\ |\ v)_k=\int_X\Td u\ol{\Td v}e^{-2k\phi}dv_X(x).
\end{equation}
We denote by $L^2(X, L^k)$ the completion of $C^\infty(X, L^k)$ with respect to $(\ |\ )_k$.  

Let
$\ddbar_{k}:C^\infty(X, L^k)\To C^\infty(X, L^k\otimes\Lambda^{0,1}T^*X)$ denote
the Cauchy-Riemann operator with values in $L^k$. 
Put $\cali{H}^0(X,L^k)=:\set{f\in C^\infty(X,L^k);\, \ddbar_kf=0}$.
Let $\pit^{(k)}:L^2(X, L^k)\To \cali{H}^0(X, L^k)$
be the Bergman projection. That is, the orthogonal projection onto $\cali{H}^0(X, L^k)$ with respect to $(\ |\ )_k$. 
Let $f\in C^\infty(X)$. Berezin-Toeplitz quantization with symbol $f$ is given by 
\begin{equation} \label{s1-e3-1}
\begin{split}
T^{(k)}_f:L^2(X,L^k)&\To L^2(X,L^k),\\
&u\To(\pit^{(k)}\circ f\circ\pit^{(k)})u.
\end{split}
\end{equation} 
Let 
\[T^{(k)}_f(x,y)\in C^\infty(X\times X,L^k_y\boxtimes L^k_x)\]
be the Schwartz kernel of $T^{(k)}_f$.
Let $s$ be a local section of $L$ over $\Td X$, where $\Td X\subset X$. Then on $\Td X\times\Td X$ we can write
\[T^{(k)}_f(x,y)=s(x)^kT^{(k)}_{f,s}(x, y)s^*(y)^k,\]
where $T^{(k)}_{f,s}(x, y)\in C^\infty(\Td X\times\Td X)$ so that for $x\in\Td X$, $u\in C^\infty_0(\Td X,L^k)$,
\begin{equation} \label{s1-e4}
\begin{split}
(T^{(k)}_{f}u)(x)&=s(x)^k\int_XT^{(k)}_{f,s}(x, y)<u(y), s^*(y)^k>dv_X(y)\\
&=s(x)^k\int_XT^{(k)}_{f,s}(x, y)\Td u(y)dv_X(y),\ \ u=s^k\Td u,\ \ \Td u\in C^\infty_0(\Td X).
\end{split}
\end{equation}
For $x=y$, we can check that the function 
$T^{(k)}_{f,s}(x, x)\in C^\infty(\Td X)$
is independent of the choices of local section $s$. We write $T^{(k)}_{f,s}(x,x)=:T^{(k)}_f(x)$ and call $T^{(k)}_f(x)$ the kernel of Berezin-Toeplitz quantization on the diagonal. 

We introduce some notations.
Let $\Omega$ be an open set of $\Real^N$. 
Let $a(x,k)\in C^\infty(\Omega)$ be a $k$-dependent function. We write 
\[a(x,k)\equiv\sum^{\infty}_{j=0}a_j(x)k^{m-j}\mod O(k^{-\infty})\ \ \mbox{on $\Omega$},\]
where $m\in\mathbb Z$, $a_j(x)\in C^\infty(\Omega)$, $j=0,1,\ldots$, if for every $N\in\mathbb N$, every $\alpha\in\mathbb N_0^{2n}$ and every compact set $K\subset\Omega$, there exists a constant $C_{N,\alpha,K}>0$ independent of $k$, such that 
\[\abs{\pr^\alpha_x\Bigr(a(x,k)-\sum^N_{j=0}k^{m-j}a_j(x)\Bigr)}\leq C_{N,\alpha,K}k^{m-N-1},\] 
$x\in K$, for $k$ large. 

Theorem~\ref{s1-t1} and Theorem~\ref{s1-t2} below are due to Ma-Marinescu~\cite[Lemma 4.6]{MM08b},~\cite[Lemma 7.2.4]{MM07},~\cite[Th.1.1]{MM08b}

\begin{thm} \label{s1-t1} 
Let $(X,\Theta)$ be a compact Hermitian manifold and $(L,h^L)\To X$ a positive line bundle. 
Let $f\in C^\infty(X)$. With the notations used above, we have 
\begin{equation} \label{s1-e5}
T^{(k)}_f(z)\equiv\sum^\infty_{j=0}b_{j,f}(z)k^{n-j}\mod O(k^{-\infty})\ \ \mbox{on $X$}, 
\end{equation} 
where $b_{j,f}(z)\in C^\infty(X)$, $j=0,1,2,\ldots$. 
\end{thm}

Let $f, g\in C^\infty(X)$. As in the discussion after \eqref{s1-e4}, we can also define the kernel of the composition $T^{(k)}_f\circ T^{(g)}_g$ on the diagonal. We write $(T^{(k)}_f\circ T^{(k)}_g)(z)$ to denote the kernel of the composition $T^{(k)}_f\circ T^{(g)}_g$ on the diagonal.

\begin{thm} \label{s1-t2} 
Let $(X,\Theta)$ be a compact Hermitian manifold and $(L,h^L)\To X$ a positive line bundle. 
Let $f, g\in C^\infty(X)$. With the notations used above, the kernel of the composition $T^{(k)}_f\circ T^{(k)}_g$ on the diagonal has an asymptotic expansion 
\begin{equation} \label{e-compI}
(T^{(k)}_f\circ T^{(k)}_g)(z)\equiv\sum^\infty_{j=0}b_{j,f,g}(z)k^{n-j}\mod O(k^{-\infty})\ \ \mbox{on $X$}, 
\end{equation} 
where $b_{j,f,g}(z)\in C^\infty(X)$, $j=0,1,2,\ldots$. 

Moreover, $T^{(k)}_f\circ T^{(k)}_g$ is a Berezin-Toeplitz quantization and it admits the asymptotic expansion 
\begin{equation} \label{e-compII}
T^{(k)}_f\circ T^{(k)}_g\equiv\sum^\infty_{j=0}k^{-j}T^{(k)}_{C_j(f,g)}\mod O(k^{-\infty}),
\end{equation} 
where $C_j$, $j=0,1,\ldots$, are bidifferential operators, in the sense that for any $m\geq0$, there exists $c_m>0$ independent of $k$ with 
\[\norm{T^{(k)}_f\circ T^{(k)}_g-\sum^m_{j=0}k^{-j}T^{(k)}_{C_j(f,g)}}\leq c_mk^{n-m-1},\] 
where $\norm{\cdot}$ denotes the operator norm on the space of bounded operators on $C^0(X,L^k)$.
\end{thm}

In \cite{MM07}, \cite{MM08b}, the asymptotic expansions \eqref{s1-e5}, \eqref{e-compI} and \eqref{e-compII}, are actually 
proved in greater generality on symplectic manifolds. 

In this paper, we give new methods for computing the coefficients of the expansions \eqref{s1-e5}, \eqref{e-compI}, \eqref{e-compII} and we calculate the first three terms of the expansions. Our purpose is to prove Theorem~\ref{s1-tmain} and Theorem~\ref{s1-tmaincomp} below. Note that we do not assume that $\omega=\Theta$. 

\begin{thm} \label{s1-tmain} 
Let $(X,\Theta)$ be a compact Hermitian manifold and $(L,h^L)\To X$ a positive line bundle. 
Let $f\in C^\infty(X)$. With the notations used above, for
\[b_{0,f}(z),\ \ b_{1,f}(z),\ \ b_{2,f}(z)\]
in \eqref{s1-e5}, we have 
\begin{equation} \label{s1-e16-I}
b_{0,f}(z)=(2\pi)^{-n}f(z)\det\dot{R}^L(z),
\end{equation}
\begin{equation} \label{s1-e16-II} 
b_{1,f}(z)=(2\pi)^{-n}f(z)\det\dot{R}^L(z)\Bigr(\frac{1}{4\pi}\hat r-\frac{1}{8\pi} r\Bigr)(z)
+(2\pi)^{-n}\det\dot{R}^L(z)\Bigr(-\frac{1}{4\pi}\triangle_{\omega}f\Bigr)(z),
\end{equation} 
\begin{equation} \label{s1-e16-III}
\begin{split}
b_{2,f}(z)&=(2\pi)^{-n}f(z)\det\dot{R}^L(z)\Bigr(\frac{1}{128\pi^2}r^2-\frac{1}{32\pi^2}r\hat r+\frac{1}{32\pi^2}(\hat r)^2
-\frac{1}{32\pi^2}\triangle_{\omega}\hat r
-\frac{1}{8\pi^2}\abs{R^{\det}_\Theta}^2_{\omega}\\
&\quad+\frac{1}{8\pi^2}\langle\,{\rm Ric\,}_{\omega}\,,R^{\det}_\Theta\,\rangle_{\omega}+\frac{1}{96\pi^2}\triangle_{\omega}r-
\frac{1}{24\pi^2}\abs{{\rm Ric\,}_{\omega}}^2_{\omega}+\frac{1}{96\pi^2}\abs{R^{TX}_{\omega}}^2_{\omega}\Bigr)(z)\\
&\quad+(2\pi)^{-n}\det\dot{R}^L(z)\Bigr(\frac{1}{16\pi^2}(\triangle_{\omega}f)(-\hat r+\frac{1}{2}r)
-\frac{1}{4\pi^2}\langle\,\ddbar\pr f\,,R^{\det}_\Theta\,\rangle_{\omega}\\
&\quad+\frac{1}{8\pi^2}\langle\,\ddbar\pr f\,,{\rm Ric\,}_{\omega}\,\rangle_{\omega}+\frac{1}{32\pi^2}\triangle^2_{\omega}f\Bigr)(z).
\end{split}
\end{equation}
We remind that $\dot{R}^L$ is given by \eqref{s1-e3} and 
\[\det\dot{R}^L(z)=\lambda_1(z)\cdots\lambda_n(z),\] 
where $\lambda_1(z),\ldots,\lambda_n(z)$ are eigenvalues of $\dot{R}^L(z)$.
\end{thm}  

\begin{thm} \label{s1-tmaincomp} 
Let $(X,\Theta)$ be a compact Hermitian manifold and $(L,h^L)\To X$ a positive line bundle. 
Let $f, g\in C^\infty(X)$. With the notations used above, for
\[b_{0,f,g}(z),\ \ b_{1,f,g}(z),\ \ b_{2,f,g}(z)\]
in \eqref{e-compI}, we have 
\begin{equation} \label{e-compmainI}
b_{0,f,g}(z)=(2\pi)^{-n}f(z)g(z)\det\dot{R}^L(z),
\end{equation}
\begin{equation} \label{e-compmainII} 
\begin{split}
b_{1,f,g}(z)&=(2\pi)^{-n}f(z)g(z)\det\dot{R}^L(z)\Bigr(\frac{1}{4\pi}\hat r-\frac{1}{8\pi} r\Bigr)(z)\\
&\quad+(2\pi)^{-n}\det\dot{R}^L(z)\Bigr(-\frac{1}{4\pi}\bigr(g\triangle_{\omega}f+f\triangle_{\omega}g\bigr)+\frac{1}{2\pi}\langle\,\ddbar f\,,\ddbar\ol g\,\rangle\,_\omega\Bigr)(z)\\
&=b_{1,fg}(z)+(2\pi)^{-n}\det\dot{R}^L(z)\Bigr(-\frac{1}{2\pi}\langle\,\pr f\,,\pr\ol g\,\rangle_\omega\Bigr)(z),
\end{split}
\end{equation} 
\begin{equation} \label{e-compmainIII}
\begin{split}
b_{2,f,g}(z)&=b_{2,fg}(z)+(2\pi)^{-n}\det\dot{R}^L(z)\Bigr(-\frac{1}{4\pi^2}\langle\,\ddbar g\wedge\pr f\,,{\rm Ric\,}_\omega\,\rangle_\omega+\frac{1}{4\pi^2}\langle\,\ddbar g\wedge\pr f\,,R^{\det}_\Theta\,\rangle_\omega \\
&\quad+\frac{1}{8\pi^2}\langle\,\pr\triangle_\omega f\,,\pr\ol g\,\rangle_\omega
+\frac{1}{8\pi^2}\langle\,\ddbar\triangle_\omega g\,,\ddbar\,\ol f\,\rangle_\omega
-\frac{1}{8\pi^2}\langle\,D^{1,0}\pr f\,,D^{1,0}\pr\ol g\,\rangle_\omega\\
&\quad-\frac{1}{4\pi^2}\langle\,\ddbar\pr f\,,\ddbar\pr\ol g\,\rangle_\omega+\frac{1}{8\pi^2}\langle\,\pr f\,,\pr\ol g\,\rangle_\omega(-\hat r+\frac{1}{2}r)\Bigr)(z).
\end{split}
\end{equation}

Moreover, for $C_0(f,g)$, $C_1(f,g)$, $C_2(f,g)$ in \eqref{e-compII}, we have 
\begin{equation} \label{e-compmainIV}
C_0(f,g)=fg,
\end{equation}
\begin{equation} \label{e-compmainV} 
C_1(f,g)=-\frac{1}{2\pi}\langle\,\pr f\,,\pr\ol g\,\rangle_\omega,
\end{equation} 
\begin{equation} \label{e-compmainVI}
C_2(f,g)=\frac{1}{8\pi^2}\langle\,D^{1,0}\pr f\,,D^{1,0}\pr\ol g\,\rangle_\omega+\frac{1}{4\pi^2}\langle\,\ddbar g\wedge\pr f\,,R^{\det}_\Theta\,\rangle_\omega.
\end{equation}
\end{thm} 

\begin{rem}\label{r-mainbis}
With the notations used above, from \eqref{s1-e16-III}, \eqref{e-compmainIII} and Proposition~\ref{s2-p1}, we can rewrite $b_{2,f,g}$ as follows: 
\begin{equation*}
\begin{split} 
b_{2,f,g}(z)&=(2\pi)^{-n}\det\dot R^L(z)f(z)g(z)\Bigr(\frac{1}{32\pi^2}\bigr(\hat r\bigr)^2-\frac{1}{32\pi^2}\hat rr\\
&\quad+\frac{1}{128\pi^2}r^2-\frac{1}{32\pi^2}\bigr(\triangle_\omega\hat r\bigr)+\frac{1}{8\pi^2}\langle\,{\rm Ric\,}_\omega\,,R^{\det}_\Theta\,\rangle_\omega-\frac{1}{8\pi^2}\abs{R^{\det}_\Theta}^2_\omega\\
&\quad+\frac{1}{96\pi^2}\bigr(\triangle_\omega r\bigr)-\frac{1}{24\pi^2}\abs{{\rm Ric\,}_\omega}^2_\omega+\frac{1}{96\pi^2}\abs{R^{TX}_\omega}^2_\omega\Bigr)(z)\\
&\quad+(2\pi)^{-n}\det\dot R^L(z)\Bigr(\bigr(\frac{1}{8\pi^2}\hat r-\frac{1}{16\pi^2}r\bigr)\langle\,\ddbar f\,,\ddbar\ol g\,\rangle_\omega+\frac{1}{16\pi^2}f\bigr(\triangle_\omega g\bigr)\bigr(-\hat r+\frac{1}{2}r\bigr)\\
&\quad+\frac{1}{16\pi^2}g\bigr(\triangle_\omega f\bigr)\bigr(-\hat r+\frac{1}{2}r\bigr)-\frac{1}{4\pi^2}f\langle\,\ddbar\pr g\,,R^{\det}_\Theta\,\rangle_\omega-\frac{1}{4\pi^2}g\langle\,\ddbar\pr f\,,R^{\det}_\Theta\,\rangle_\omega\\
&\quad+\frac{1}{8\pi^2}f\langle\,\ddbar\pr g\,,{\rm Ric\,}_\omega\,\rangle_\omega
+\frac{1}{8\pi^2}g\langle\,\ddbar\pr f\,,{\rm Ric\,}_\omega\,\rangle_\omega
+\frac{1}{4\pi^2}\langle\,\ddbar f\wedge\pr g\,,{\rm Ric\,}_\omega\,\rangle_\omega\\
&\quad-\frac{1}{4\pi^2}\langle\,\ddbar f\wedge\pr g\,,R^{\det}_\Theta\,\rangle_\omega
-\frac{1}{8\pi^2}\langle\,\ddbar f\,,\ddbar\triangle_\omega\ol g\,\rangle_\omega-\frac{1}{8\pi^2}\langle\,\ddbar\triangle_\omega f\,,\ddbar\ol g\,\rangle_\omega\\
&\quad+\frac{1}{8\pi^2}\langle\,D^{0,1}\ddbar f\,,D^{0,1}\ddbar\ol g\,\rangle_\omega+\frac{1}{32\pi^2}f\bigr(\triangle^2_\omega g\bigr)+\frac{1}{32\pi^2}g\bigr(\triangle^2_\omega f\bigr)\\
&\quad+\frac{1}{16\pi^2}\bigr(\triangle_\omega g\bigr)\bigr(\triangle_\omega f\bigr)\Bigr)(z).
\end{split}
\end{equation*} 
\end{rem} 

\begin{rem}\label{r-mainbisaddb}
(I)\,
In \cite{MM10}, Ma-Marinescu calculated the coefficients $b_{0,f}$, $b_{1,f}$, $b_{2,f}$, $b_{0,f,g}$, $b_{1,f,g}$, $b_{2,f,g}$ and $C_0(f,g)$, $C_1(f,g)$, $C_2(f,g)$, in the case when $\omega=\Theta$ 
and in the presence of a twisting vector bundle. In~\cite[section 2.7]{MM11}, they observed that one can reduce the 
calculation in the case when $\omega\neq\Theta$ to the case when $\omega=\Theta$ by the following trick. Let $b_{j,f}$ be as in \eqref{s1-e5} corresponding to the case $\omega\neq\Theta$. Let $E$ be the trivial line bundle over $X$ (i.e. $E=\Complex$) and endow $E$ with the metric $\abs{1}^2=:(2\pi)^{n}\bigr(\det\dot R^L\bigr)^{-1}$. Endow $X$ with the Hermitian metric $\omega=\Theta$ and consider the $L^2$ inner product induced by the metric of $E$, $\omega$, and $h^{L^k}$ as in \eqref{e-inner} and let $T^{(k)}_{f,E}$ be the Berezin-Toeplitz quantization with values in $L^k\otimes E$ and let $b_{j,f,E}$ be as in \eqref{s1-e5}. Then by~\cite[(2.110)]{MM11}, we have 
\[b_{j,f}=(2\pi)^{-n}\det\dot R^Lb_{j,f,E},\ \ j=0,1,\ldots.\]
One can check that this formulas coincide with those from Theorem~\ref{s1-tmain}. Note also that the formulas for the coefficients $C_j(f,g)$
are the same in the case when $\omega=\Theta$ and in the case when $\omega\neq\Theta$, see also ~\cite[(2.110)]{MM11}.
\\[2pt]
\
(II)\,Note that we can also include a twisting bundle in our computation but due to the fact we already consider the case $\omega\neq\Theta$ the formulas become quite long.

\end{rem}

We say that $X$ is polarized if $\omega=\Theta$. We can check that if $X$ is polarized then 
\begin{equation} \label{s1-e17}
\det\dot{R}^L=(2\pi)^n,\ \ r=\hat r,\ \ {\rm Ric\,}_\omega=R^{\det}_\Theta.
\end{equation}
From this observation and Theorem~\ref{s1-tmain}, Theorem~\ref{s1-tmaincomp}, we deduce the following results of Ma-Marinescu~\cite{MM10}

\begin{cor} \label{s1-c1}
If $X$ is polarized, then for $b_{0,f}(z), b_{1,f}(z), b_{2,f}(z)$ in \eqref{s1-e5}, we have
\begin{equation} \label{s1-e18}
\begin{split} 
b_{0,f}(z)&=f(z),\\
b_{1,f}(z)&=\frac{1}{8\pi}(r f)(z)-\frac{1}{4\pi}(\triangle_{\omega}f)(z),\\
b_{2,f}(z)&=f(z)\Bigr(\frac{1}{128\pi^2}r^2-\frac{1}{48\pi^2}\triangle_{\omega}r
-\frac{1}{24\pi^2}\abs{{\rm Ric\,}_{\omega}}^2_{\omega}+\frac{1}{96\pi^2}\abs{R^{TX}_{\omega}}^2_{\omega}\Bigr)(z)\\
&\quad+\Bigr(-\frac{1}{32\pi^2}(\triangle_{\omega}f)r-\frac{1}{8\pi^2}\langle\,\ddbar\pr f\,,{\rm Ric\,}_{\omega}\,\rangle_{\omega}
+\frac{1}{32\pi^2}\triangle^2_{\omega}f\Bigr)(z).
\end{split}
\end{equation} 

Moreover, for $b_{0,f,g}(z), b_{1,f,g}(z), b_{2,f,g}(z)$ in \eqref{e-compI}, we have 
\begin{equation} \label{s1-e19}
\begin{split} 
b_{0,f,g}(z)&=f(z)g(z),\\
b_{1,f,g}(z)&=f(z)g(z)\frac{1}{8\pi}r(z)+\frac{1}{2\pi}\langle\,\ddbar f\,,\ddbar\ol g\,\rangle_\omega(z)
-\frac{1}{4\pi}\bigr(g\triangle_{\omega}f+f\triangle_{\omega}g\bigr)(z),\\
b_{2,f,g}(z)&=b_{2,fg}(z)+\frac{1}{8\pi^2}\langle\,\pr\triangle_\omega f\,,\pr\ol g\,\rangle_\omega(z)
+\frac{1}{8\pi^2}\langle\,\ddbar\triangle_\omega g\,,\ddbar\,\ol f\,\rangle_\omega(z)\\
&\quad-\frac{1}{8\pi^2}\langle\,D^{1,0}\pr f\,,D^{1,0}\pr\ol g\,\rangle_\omega(z)
-\frac{1}{4\pi^2}\langle\,\ddbar\pr f\,,\ddbar\pr\ol g\,\rangle_\omega(z)\\
&\quad-\frac{1}{16\pi^2}\langle\,\pr f\,,\pr\ol g\,\rangle_\omega(z)r(z).
\end{split}
\end{equation} 
\end{cor}

\section{The Taylor expansions of some global functions at a given point}

In this section, we will use the same notations as section 1. For a given point $p\in X$, we may take local holomorphic 
coordinates $z=(z_1,z_2,\ldots,z_n)$ and local trivializing section $s$ of $L$ defined in some small 
open neighborhood of $p$ such that 
\begin{equation} \label{s2-eI}
\begin{split}
&z(p)=0,\\
&\phi(z)=\sum^n_{j=1}\lambda_j\abs{z_j}^2+\phi_1(z),\\
&\phi_1(z)=O(\abs{z})^4),\ \ \frac{\pr^{\abs{\alpha}+\abs{\beta}}\phi_1}{\pr z^\alpha\pr\ol z^\beta}(0)=0\ \ \mbox{if $\abs{\alpha}\leq1$ or $\abs{\beta}\leq1$},\ \ \alpha, \beta\in\mathbb N_0^n,\\
&\Theta(z)=\sqrt{-1}\sum^n_{j=1}dz_j\wedge d\ol z_j+O(\abs{z}).
\end{split}
\end{equation}
(This is always possible. See Ruan~\cite{Ruan98}.)
In this section, we work with this local coordinates $z$ and we identify $p$ with the point $z=0$. 

From \eqref{s2-eI}, we can check that $2\lambda_1,\ldots,2\lambda_n$, are eigenvalues of $\dot{R}^L(0)$ and 
\begin{equation} \label{s2-eII}
\omega=\frac{\sqrt{-1}}{\pi}\sum^n_{j=1}\lambda_jdz_j\wedge d\ol z_j+\frac{\sqrt{-1}}{\pi}\sum^n_{j,k=1}\frac{\pr^2\phi_1}{\pr z_j\pr\ol z_k}dz_j\wedge d\ol z_k.
\end{equation}

The following lemma follows from some straightforward but elementary computations. We omit the proof. 

\begin{lem} \label{s2-lI}
We have 
\begin{align} 
&\mathcal{R}_{j,k}(z)=\frac{1}{\lambda_j}\sum^n_{s,t=1}\frac{\pr^4\phi}{\pr\ol z_j\pr z_k\pr\ol z_s\pr z_t}(z)d\ol z_s\wedge dz_t+O(\abs{z}^2),\ \ j,k=1,\ldots,n.\label{s2-eIII}\\
&\triangle_\omega=-2\pi\sum^n_{j=1}\frac{1}{\lambda_j}\frac{\pr^2}{\pr z_j\pr\ol z_j}+2\pi\sum^n_{j,k=1}\frac{1}{\lambda_j\lambda_k}\frac{\pr^2\phi_1}{\pr\ol z_j\pr z_k}\frac{\pr^2}{\pr z_j\pr\ol z_k}+O(\abs{z}^3).\label{s2-eIV}\\
&\log V_\omega=\sum^n_{s=1}\log\frac{\pr^2\phi}{\pr\ol z_s\pr z_s}-\sum_{s<t,1\leq s,t\leq n}\frac{1}{\lambda_s\lambda_t}\frac{\pr^2\phi}{\pr z_s\pr\ol z_t}\frac{\pr^2\phi}{\pr\ol z_s\pr z_t}+O(\abs{z}^5).\label{s2-eV}
\end{align} 
(We remind that $\mathcal{R}_{j,k}$ and $V_\omega$ are given by \eqref{s1-e13} and \eqref{s1-e10} respectively.)
\end{lem}

Put 
\begin{equation} \label{s2-eV-I}
\triangle_0=\sum^n_{j=1}\frac{1}{\lambda_j}\frac{\pr^2}{\pr\ol z_j\pr z_j}.
\end{equation} 
The following theorem also follows from some straightforward computations. We only sketch the proof 

\begin{thm} \label{s2-tI}
We have 
\begin{align} 
&\abs{R^{TX}_\omega}^2_\omega(0)=\pi^2\sum_{1\leq s,t,j,k\leq n}\frac{1}{\lambda_t\lambda_s\lambda_j\lambda_k}
\abs{\frac{\pr^4\phi}{\pr\ol z_j\pr z_s\pr\ol z_t\pr z_k}(0)}^2.\label{s2-eVI}\\
&{\rm Ric\,}_\omega(0)=-\sum^n_{s,t,j=1}\frac{1}{\lambda_j}\frac{\pr^4\phi}{\pr\ol z_j\pr z_j\pr\ol z_s\pr z_t}(0)d\ol z_s\wedge dz_t.\label{s2-eVII}\\
&R^{\det}_\Theta(0)=-\sum^n_{s,t=1}\frac{\pr^2\log V_\Theta}{\pr\ol z_s\pr z_t}(0)d\ol z_s\wedge dz_t.\label{s2-eVIII}\\  
&r(z)=-2\pi\sum^n_{s,t=1}\frac{1}{\lambda_s\lambda_t}\frac{\pr^4\phi}{\pr\ol z_s\pr z_s\pr\ol z_t\pr z_t}(z)+O(\abs{z}^2)=-2\pi(\triangle_0^2\phi)(z)+O(\abs{z}^2).\label{s2-eIX}\\
&\hat r(z)=-2\pi\sum^n_{j=1}\frac{1}{\lambda_j}\Bigr(-\abs{\frac{\pr V_\Theta}{\pr z_j}}^2+\frac{\pr^2V_\Theta}{\pr\ol z_j\pr z_j}\Bigr)(z)+O(\abs{z}^2)=-2\pi(\triangle_0\log V_\Theta)(z)+O(\abs{z}^2).\label{s2-eX}\\
&(\triangle_\omega r)(0)=4\pi^2(\triangle^3_0\phi)(0)-8\abs{{\rm Ric\,}_\omega}^2_\omega(0)-4\abs{R^{TX}_\omega}^2_\omega(0).\label{s2-eXI}\\
&(\triangle_\omega\hat r)(0)=4\pi^2(\triangle^2_0\log V_\Theta)(0)-4\langle\,R^{\det}_\Theta\,,{\rm Ric\,}_\omega\,\rangle_\omega(0).\label{s2-eXII} 
\end{align} 
Let $f\in C^\infty(X)$. Then, 
\begin{align} 
&(D^{1,0}\pr f)(0)=\sum^n_{j,k=1}\frac{\pr^2f}{\pr z_j\pr z_k}(0)dz_k\otimes dz_j.\label{s2-ecomp0.1}\\
&(D^{0,1}\ddbar f)(0)=\sum^n_{j,k=1}\frac{\pr^2f}{\pr\ol z_j\pr\ol z_k}(0)d\ol z_k\otimes d\ol z_j.\label{s2-ecomp0.2}\\
&(\triangle_\omega f)(0)=-2\pi(\triangle_0 f)(0).\label{s2-eXII-I}\\ 
&(\pr\triangle_\omega f)(0)=-2\pi(\pr\triangle_0f)(0).\label{s2-ecomp0.3}\\
&(\ddbar\triangle_\omega f)(0)=-2\pi(\ddbar\triangle_0f)(0).\label{s2-ecomp0.4}\\
&(\triangle_\omega^2f)(0)=4\pi^2(\triangle^2_0f)(0)+4\langle\,\ddbar\pr f\,,{\rm Ric\,}_\omega\,\rangle_\omega(0). 
\label{s2-eXII-bis} 
\end{align} 
\end{thm}

\begin{proof}
Let $e_j=e_j(z)$, $j=1,\ldots,n$, be an orthonormal frame for $T^{1,0}_zX$ such that 
\begin{equation} \label{s2-eXIII}
e_j(z)=\sqrt{\frac{\pi}{\lambda_j}}\frac{\pr}{\pr z_j}+O(\abs{z}^2),\ \ j=1,\ldots,n.
\end{equation} 
(It is easy to see that this is always possible.)
From \eqref{s2-eIII}, it is not difficult to see that 
\[R^{TX}_\omega(\ol e_s,e_t)e_k=\frac{\pi\sqrt{\pi}}{\sqrt{\lambda_s\lambda_t\lambda_k}}\sum^n_{j=1}\frac{1}{\lambda_j}\frac{\pr^4\phi}{\pr\ol z_s\pr z_t\pr\ol z_j\pr z_k}(z)\frac{\pr}{\pr z_j}+O(\abs{z}^2),\ \ \forall\ \ 1\leq s,t,k\leq n.\]
Thus, 
\begin{equation} \label{s2-eXIV}
\langle\,R^{TX}_\omega(\ol e_s,e_t)e_k\,,e_j\,\rangle_\omega=\frac{\pi}{\sqrt{\lambda_s\lambda_t\lambda_k\lambda_j}}
\frac{\pr^4\phi}{\pr\ol z_j\pr z_k\pr\ol z_s\pr z_t}(z)+O(\abs{z}^2).
\end{equation} 
From \eqref{s2-eXIV} and \eqref{s1-e14}, \eqref{s2-eVI} follows. 

Similarly, from \eqref{s2-eIII}, we can check that 
\begin{equation} \label{s2-eXV}
\begin{split}
\langle\,{\rm Ric\,}_\omega\,,\frac{\pr}{\pr\ol z_s}\wedge\frac{\pr}{\pr z_t}\,\rangle&=
-\sum^n_{j=1}\langle\,R^{TX}_\omega(\frac{\pr}{\pr\ol z_s},e_j)\frac{\pr}{\pr z_t}\,,e_j\,\rangle_\omega\\
&=-\sum^n_{j=1}\frac{1}{\lambda_j}\frac{\pr^4\phi}{\pr\ol z_j\pr z_j\pr\ol z_s\pr z_t}(z)+O(\abs{z}^2).
\end{split}
\end{equation} 
From \eqref{s2-eXV}, \eqref{s2-eVII} follows. 

From \eqref{s2-eIV}, \eqref{s2-eV} and the definitions of $R^{\det}_\Theta$, $r$, $\hat r$ (see \eqref{s1-e12}, \eqref{s1-e11}), 
we can easily get \eqref{s2-eVIII}, \eqref{s2-eIX}, \eqref{s2-eX}. 

Now, we prove \eqref{s2-eXII}. From \eqref{s2-eIV}, we can check that 
\begin{equation} \label{s2-eXVI}
\begin{split}
(\triangle_\omega\hat r)(0)&=(\triangle^2_\omega\log V_\Theta)(0)\\
&=\Bigr(\triangle_\omega\circ\Bigr(-2\pi(\triangle_0\log V_\Theta)
+2\pi\sum^n_{j,k=1}\frac{1}{\lambda_j\lambda_k}\frac{\pr^2\phi_1}{\pr\ol z_j\pr z_k}\frac{\pr^2\log V_\Theta}{\pr z_j\pr\ol z_k}\Bigr)\Bigr)(0)\\
&=\Bigr(-2\pi\triangle_0\circ\Bigr(-2\pi(\triangle_0\log V_\Theta)
+2\pi\sum^n_{j,k=1}\frac{1}{\lambda_j\lambda_k}\frac{\pr^2\phi_1}{\pr\ol z_j\pr z_k}\frac{\pr^2\log V_\Theta}{\pr z_j\pr\ol z_k}\Bigr)\Bigr)(0)\\
&=4\pi^2(\triangle^2_0\log V_\Theta)(0)-4\pi^2\sum^n_{j,k,s=1}\frac{1}{\lambda_j\lambda_k\lambda_s}\frac{\pr^4\phi}{\pr\ol z_j\pr z_k\pr\ol z_s\pr z_s}(0)\frac{\pr^2\log V_\Theta}{\pr z_j\ol z_k}(0).
\end{split}
\end{equation} 
In view of \eqref{s2-eVIII} and \eqref{s2-eVII}, we see that 
\begin{equation} \label{s2-eXVII}
\langle\,{\rm Ric\,}_\omega\,,R^{\det}_\Theta\,\rangle_\omega(0)=\sum^n_{j,k,s=1}\frac{\pi^2}{\lambda_j\lambda_k\lambda_s}\frac{\pr^4\phi}{\pr\ol z_j\pr z_j\pr\ol z_k\pr z_s}(0)\frac{\pr^2\log V_\Theta}{\pr z_k\pr\ol z_s}(0).
\end{equation}
From this observation and \eqref{s2-eXVI}, \eqref{s2-eXII} follows. 

Now, we prove \eqref{s2-eXI}. We can repeat the procedure as \eqref{s2-eXVI} and conclude that 
\begin{equation} \label{s2-eXVIII}
\begin{split}
(\triangle_\omega r)(0)&=(\triangle^2_\omega\log V_\omega)(0)\\
&=4\pi^2(\triangle^2_0\log V_\omega)(0)
-4\pi^2\sum^n_{j,k,s=1}\frac{1}{\lambda_j\lambda_k\lambda_s}\frac{\pr^4\phi}{\pr\ol z_j\pr z_k\pr\ol z_s\pr z_s}(0)\frac{\pr^2\log V_\omega}{\pr z_j\ol z_k}(0).
\end{split}
\end{equation} 
As \eqref{s2-eXVII}, we have 
\[\langle\,{\rm Ric\,}_\omega\,,-\ddbar\pr\log V_\omega\,\rangle_\omega(0)=\sum^n_{j,k,s=1}\frac{\pi^2}{\lambda_j\lambda_k\lambda_s}\frac{\pr^4\phi}{\pr\ol z_j\pr z_k\pr\ol z_s\pr z_s}(0)\frac{\pr^2\log V_\omega}{\pr z_j\pr\ol z_k}(0).\] 
Note that $-\ddbar\pr\log V_\omega={\rm Ric\,}_\omega$. Thus, 
\begin{equation} \label{s2-eXIX} 
\abs{{\rm Ric\,}_\omega}^2_\omega(0)=\sum^n_{j,k,s=1}\frac{\pi^2}{\lambda_j\lambda_k\lambda_s}\frac{\pr^4\phi}{\pr\ol z_j\pr z_k\pr\ol z_s\pr z_s}(0)\frac{\pr^2\log V_\omega}{\pr z_j\pr\ol z_k}(0).
\end{equation} 
From \eqref{s2-eV}, it is straightforward to see that 
\begin{equation} \label{s2-eXX}
\begin{split}
(\triangle^2_0\log V_\omega)(0)&=(\triangle^3_0\phi)(0)-\sum^n_{s,t,j,k=1}\frac{1}{\lambda_s\lambda_t\lambda_j\lambda_k}\frac{\pr^4\phi}{\pr z_s\pr\ol z_t\pr\ol z_k\pr z_k}(0)
\frac{\pr^4\phi}{\pr\ol z_s\pr z_t\pr\ol z_j\pr z_j}(0)\\
&\quad-\sum^n_{s,t,j,k=1}\frac{1}{\lambda_s\lambda_t\lambda_j\lambda_k}\abs{\frac{\pr^4\phi}{\pr z_s\pr\ol z_t\pr\ol z_k\pr z_j}(0)}^2.
\end{split}
\end{equation} 
From \eqref{s2-eVII} and \eqref{s2-eVI}, we see that 
\begin{equation} \label{s2-eXXI} 
\sum^n_{s,t,j,k=1}\frac{1}{\lambda_s\lambda_t\lambda_j\lambda_k}\frac{\pr^4\phi}{\pr z_s\pr\ol z_t\pr\ol z_k\pr z_k}(0)
\frac{\pr^4\phi}{\pr\ol z_s\pr z_t\pr\ol z_j\pr z_j}(0)=\frac{1}{\pi^2}\abs{{\rm Ric\,}_\omega}^2_\omega(0)
\end{equation} 
and 
\begin{equation} \label{s2-eXXII}
\sum^n_{s,t,j,k=1}\frac{1}{\lambda_s\lambda_t\lambda_j\lambda_k}\abs{\frac{\pr^4\phi}{\pr z_s\pr\ol z_t\pr\ol z_k\pr z_j}(0)}^2=\frac{1}{\pi^2}\abs{R^{TX}_\omega}^2_\omega(0).
\end{equation} 
From \eqref{s2-eXXI}, \eqref{s2-eXXII} and \eqref{s2-eXX}, we obtain 
\begin{equation} \label{s2-eXXIII}
(\triangle^2_0\log V_\omega)(0)=(\triangle^3_0\phi)(0)-\frac{1}{\pi^2}\abs{{\rm Ric\,}_\omega}^2_\omega(0)-\frac{1}{\pi^2}\abs{R^{TX}_\omega}^2_\omega(0).
\end{equation} 
Combining \eqref{s2-eXXIII} with \eqref{s2-eXIX} and \eqref{s2-eXVIII}, \eqref{s2-eXI} follows. 

In view of \eqref{s2-eI}, \eqref{s2-eIV}, we see that 
\[\triangle_\omega f=-2\pi\triangle_0f+O(\abs{z}^2).\] 
From this, \eqref{s2-eXII-I}, \eqref{s2-ecomp0.3} and \eqref{s2-ecomp0.4} follows. 

The proof of \eqref{s2-eXII-bis} is essentially the same as the proof of \eqref{s2-eXII}. 

From \eqref{s2-eII}, it is not difficult to see that $\alpha_{j,k}=O(\abs{z})$, $j,k=1,\ldots,n$, where $\alpha_{j,k}$, 
$j,k=1,\ldots,n$, are as in \eqref{s1-e15-Ibis}. From this observation, \eqref{s2-ecomp0.1} and \eqref{s2-ecomp0.2} follows.
\end{proof}

From Lemma~\ref{s2-lI} and Theorem~\ref{s2-tI}, we deduce 

\begin{cor} \label{s2-cI}
With the notations used before, we have 
\begin{equation} \label{s2-eVII-I}
\abs{{\rm Ric\,}_\omega}^2_\omega(0)=\pi^2\sum_{1\leq s,t,j,k\leq n}\frac{1}{\lambda_t\lambda_s\lambda_j\lambda_k}
\frac{\pr^4\phi}{\pr\ol z_s\pr z_s\pr\ol z_t\pr z_k}(0)\frac{\pr^4\phi}{\pr\ol z_j\pr z_j\pr z_t\pr\ol z_k}(0).
\end{equation} 
\begin{equation}\label{s2-eVIII-bis}
\abs{R^{\det}_\Theta}^2_\omega(0)=\pi^2\sum^n_{s,t=1}\frac{1}{\lambda_s\lambda_t}\abs{\frac{\pr^2\log V_\Theta}{\pr z_s\pr\ol z_t}(0)}^2.
\end{equation} 
\begin{equation}\label{s2-eVIII-bisb}
\langle\,R^{\det}_\Theta\,,{\rm Ric\,}_\omega\,\rangle_\omega(0)=
\pi^2\sum^n_{j,k,s=1}\frac{1}{\lambda_j\lambda_k\lambda_s}\frac{\pr^4\phi}{\pr\ol z_j\pr z_j\pr\ol z_k\pr z_s}(0)\frac{\pr^2\log V_\Theta}{\pr z_k\pr\ol z_s}(0).
\end{equation}
\begin{equation}\label{s2-eIX-I}
r^2(0)=4\pi^2\sum_{1\leq s,t,j,k\leq n}\frac{1}{\lambda_t\lambda_s\lambda_j\lambda_k}
\frac{\pr^4\phi}{\pr\ol z_t\pr z_t\pr\ol z_s\pr z_s}(0)\frac{\pr^4\phi}{\pr\ol z_k\pr z_k\pr\ol z_j\pr z_j}(0).
\end{equation} 
Let $f,g\in C^\infty(X)$. Then, 
\begin{equation} \label{s2-eXII-II}
\langle\,\ddbar f\,,\ddbar\hat r\,\rangle_\omega(0)=-2\pi^2\sum^n_{s=1}\frac{1}{\lambda_s}\frac{\pr f}{\pr\ol z_s}(0)\frac{\pr}{\pr z_s}(\triangle_0\log V_\Theta)(0). 
\end{equation} 
\begin{equation}\label{s2-eXII-III}
\langle\,\pr f\,,\pr\hat r\,\rangle_\omega(0)=-2\pi^2\sum^n_{s=1}\frac{1}{\lambda_s}\frac{\pr f}{\pr z_s}(0)\frac{\pr}{\pr \ol z_s}(\triangle_0\log V_\Theta)(0). 
\end{equation} 
\begin{equation}\label{s2-eXII-IIIbisb}
\langle\,\ddbar f\,,\ddbar r\,\rangle_\omega(0)=-2\pi^2\sum^n_{j,k,s=1}\frac{1}{\lambda_j\lambda_k\lambda_s}\frac{\pr^5\phi}{\pr\ol z_j\pr z_j\pr\ol z_k\pr z_k\pr z_s}(0)\frac{\pr f}{\pr\ol z_s}(0).
\end{equation}
\begin{equation} \label{s2-eXII-IIIbisbb}
\langle\,\pr f\,,\pr r\,\rangle_\omega(0)=-2\pi^2\sum^n_{j,k,s=1}\frac{1}{\lambda_j\lambda_k\lambda_s}\frac{\pr^5\phi}{\pr\ol z_j\pr z_j\pr\ol z_k\pr z_k\pr\ol z_s}(0)\frac{\pr f}{\pr z_s}(0).
\end{equation}
\begin{equation} \label{s2-eXII-IIIbis}
\langle\,\ddbar\pr f\,,R^{\det}_\Theta\,\rangle_\omega(0)=-\pi^2\sum^n_{s,t=1}\frac{1}{\lambda_s\lambda_t}\frac{\pr^2f}{\pr\ol z_s\pr z_t}(0)\frac{\pr^2\log V_\Theta}{\pr z_s\pr\ol z_t}(0).
\end{equation} 
\begin{equation}\label{s2-eXII-IV}
\langle\,\ddbar\pr f\,,{\rm Ric\,}_\omega\,\rangle_\omega(0)=-\pi^2\sum^n_{j,k,s=1}\frac{1}{\lambda_j\lambda_k\lambda_s}\frac{\pr^4\phi}{\pr\ol z_j\pr z_j\pr\ol z_k\pr z_s}(0)\frac{\pr^2f}{\pr z_k\pr\ol z_s}(0).
\end{equation}
\begin{equation}\label{s2-ecomppreb0}
\langle\,\ddbar g\wedge\pr f\,,{\rm Ric\,}_\omega\,\rangle_\omega=-\pi^2\sum^n_{j,k,s=1}\frac{1}{\lambda_j\lambda_k\lambda_s}\frac{\pr^4\phi}{\pr\ol z_j\pr z_k\pr z_s\pr\ol z_s}(0)\frac{\pr f}{\pr z_j}(0)\frac{\pr g}{\pr\ol z_k}(0).
\end{equation} 
\begin{equation}\label{s2-ecomppre0}
\langle\,\ddbar\pr f\,,\ddbar\pr\ol g\,\rangle_\omega(0)=\pi^2\sum^n_{j,k=1}\frac{1}{\lambda_j\lambda_k}\frac{\pr^2f}{\pr\ol z_j\pr z_k}(0)\frac{\pr^2g}{\pr z_j\pr\ol z_k}(0).
\end{equation} 
\begin{equation}\label{s2-ecomppreI}
\langle\,D^{1,0}\pr f\,,D^{1,0}\pr\ol g\,\rangle_\omega(0)=\pi^2\sum^n_{j,k=1}\frac{1}{\lambda_j\lambda_k}\frac{\pr^2f}{\pr z_j\pr z_k}(0)
\frac{\pr^2g}{\pr\ol z_j\pr\ol z_k}(0).
\end{equation} 
\begin{equation}\label{s2-ecomppreII}
\langle\,D^{0,1}\ddbar f\,,D^{0,1}\ddbar\ol g\,\rangle_\omega(0)=\pi^2\sum^n_{j,k=1}\frac{1}{\lambda_j\lambda_k}\frac{\pr^2f}{\pr\ol z_j\pr\ol z_k}(0)
\frac{\pr^2g}{\pr z_j\pr z_k}(0).
\end{equation}
\begin{equation}\label{s2-ecomppreIII}
\langle\,\pr\triangle_\omega f\,,\pr\ol g\,\rangle_\omega(0)=-2\pi^2\sum^n_{j,k=1}\frac{1}{\lambda_j\lambda_k}\frac{\pr^3f}{\pr z_j\pr\ol z_j\pr z_k}(0)
\frac{\pr g}{\pr\ol z_k}(0).
\end{equation}
\begin{equation}\label{s2-ecomppreIV}
\langle\,\ddbar\triangle_\omega f\,,\ddbar\ol g\,\rangle_\omega(0)=-2\pi^2\sum^n_{j,k=1}\frac{1}{\lambda_j\lambda_k}\frac{\pr^3f}{\pr z_j\pr\ol z_j\pr \ol z_k}(0)\frac{\pr g}{\pr z_k}(0).
\end{equation}
\end{cor} 

Formula \eqref{s2-ecompII} below appears in~\cite[(5.80)]{MM10}. For the convenience of the reader we give the details here.

\begin{prop} \label{s2-p1}
Let $f, g\in C^\infty(X)$. We have
\begin{equation} \label{s2-ecompI}
\triangle_\omega(fg)=(\triangle_\omega f)g+(\triangle_\omega g)f-2\langle\,\pr f\,,\pr\ol g\,\rangle_\omega-2\langle\,\ddbar f\,,\ddbar\ol g\,\rangle_\omega
\end{equation} 
and 
\begin{equation} \label{s2-ecompII}
\begin{split}
\triangle^2_\omega(fg)&=(\triangle^2_\omega f)g+(\triangle^2_\omega g)f-4\langle\,\pr\triangle_\omega f\,,\pr\ol g\,\rangle_\omega-4\langle\,\ddbar\triangle_\omega f\,,\ddbar\ol g\,\rangle_\omega
-4\langle\,\pr\triangle_\omega g\,,\pr\ol f\,\rangle_\omega\\
&\quad-4\langle\,\ddbar\triangle_\omega g\,,\ddbar\,\ol f\,\rangle_\omega
+2(\triangle_\omega f)(\triangle_\omega g)
+8\langle\,\ddbar\pr f\,,\ddbar\pr\ol g\,\rangle_\omega+4\langle\,D^{0,1}\ddbar f\,,D^{0,1}\ddbar\ol g\,\rangle_\omega\\
&\quad
+4\langle\,D^{1,0}\pr f\,,D^{1,0}\pr\ol g\,\rangle_\omega+4\langle\,\ddbar f\wedge\pr g\,,{\rm Ric\,}_\omega\,\rangle_\omega
+4\langle\,\ddbar g\wedge\pr f\,,{\rm Ric\,}_\omega\,\rangle_\omega.
\end{split}
\end{equation}
\end{prop}

\begin{proof}
From \eqref{s2-eII} and \eqref{s2-eIV}, we can check that 
\begin{equation} \label{s2-ecompIII}
\begin{split}
\triangle_\omega(fg)(0)&=-2\pi\Bigr((\triangle_0f)g+(\triangle_0g)f\Bigr)(0)-2\pi\sum^n_{j=1}
\frac{1}{\lambda_j}\Bigr(\frac{\pr f}{\pr z_j}(0)\frac{\pr g}{\pr\ol z_j}(0)+\frac{\pr f}{\pr\ol z_j}(0)\frac{\pr g}{\pr z_j}(0)\Bigr)\\
&=(\triangle_\omega f)(0)g(0)+(\triangle_\omega g)(0)f(0)-2\langle\,\pr f\,,\pr\ol g\,\rangle_\omega(0)-2\langle\,\ddbar f\,,\ddbar\ol g\,\rangle_\omega(0).
\end{split}
\end{equation} 
\eqref{s2-ecompI} follows. 

Now, we prove \eqref{s2-ecompII}. In view of \eqref{s2-eXII-bis}, we have 
\begin{equation} \label{s2-ecompIV}
\begin{split}
\triangle^2_\omega(fg)(0)&=4\pi^2\triangle^2_0(fg)(0)+4\langle\,\ddbar\pr(fg)\,,{\rm Ric\,}_\omega\,\rangle_\omega(0)\\
&=4\pi^2\triangle^2_0(fg)(0)+4\langle\,g\ddbar\pr f+f\ddbar\pr g+\ddbar f\wedge\pr g+\ddbar g\wedge\pr f\,,{\rm Ric\,}_\omega\,\rangle_\omega(0).
\end{split}
\end{equation} 
It is straightforward to see that 
\begin{equation} \label{s2-ecompV} 
\begin{split}
\triangle^2_0(fg)(0)&=(\triangle^2_0f)(0)g(0)+(\triangle^2_0g)(0)f(0)+2\sum^n_{j,k=1}\frac{1}{\lambda_j\lambda_k}\frac{\pr^3f}{\pr z_j\pr\ol z_j\pr z_k}(0)\frac{\pr g}{\pr\ol z_k}(0)\\
&\quad+2\sum^n_{j,k=1}\frac{1}{\lambda_j\lambda_k}\frac{\pr^3f}{\pr z_j\pr\ol z_j\pr\ol z_k}(0)\frac{\pr g}{\pr z_k}(0)
+2\sum^n_{j,k=1}\frac{1}{\lambda_j\lambda_k}\frac{\pr^3g}{\pr z_j\pr\ol z_j\pr z_k}(0)\frac{\pr f}{\pr\ol z_k}(0)\\
&\quad+2\sum^n_{j,k=1}\frac{1}{\lambda_j\lambda_k}\frac{\pr^3g}{\pr z_j\pr\ol z_j\pr\ol z_k}(0)\frac{\pr g}{\pr z_k}(0)
+2(\triangle_0f)(0)(\triangle_0g)(0)\\
&\quad+2\sum^n_{j,k=1}\frac{1}{\lambda_j\lambda_k}\frac{\pr^2f}{\pr z_k\pr\ol z_j}(0)\frac{\pr^2g}{\pr\ol z_k\pr z_j}(0)+\sum^n_{j,k=1}\frac{1}{\lambda_j\lambda_k}\frac{\pr^2f}{\pr z_k\pr z_j}(0)\frac{\pr^2g}{\pr\ol z_k\pr\ol z_j}(0)\\
&\quad+\sum^n_{j,k=1}\frac{1}{\lambda_j\lambda_k}\frac{\pr^2f}{\pr\ol z_k\pr\ol z_j}(0)\frac{\pr^2g}{\pr z_k\pr z_j}(0).
\end{split}
\end{equation}
Combining \eqref{s2-ecompV} with \eqref{s2-ecomppre0}, \eqref{s2-ecomppreI}, \eqref{s2-ecomppreII}, \eqref{s2-ecomppreIII} 
and \eqref{s2-ecomppreIV}, we obtain
\begin{equation} \label{s2-ecompVI} 
\begin{split}
\triangle^2_0(fg)(0)&=(\triangle^2_0f)(0)g(0)+(\triangle^2_0g)(0)f(0)-
\frac{1}{\pi^2}\langle\,\pr\triangle_\omega f\,,\pr\ol g\,\rangle_\omega(0)\\
&\quad-\frac{1}{\pi^2}\langle\,\ddbar\triangle_\omega f\,,\ddbar\ol g\,\rangle_\omega(0)
-\frac{1}{\pi^2}\langle\,\pr\triangle_\omega g\,,\pr\ol f\,\rangle_\omega(0)
-\frac{1}{\pi^2}\langle\,\ddbar\triangle_\omega g\,,\ddbar\,\ol f\,\rangle_\omega(0)\\
&\quad+2(\triangle_0f)(0)(\triangle_0g)(0)
+\frac{2}{\pi^2}\langle\,\ddbar\pr f\,,\ddbar\pr\ol g\,\rangle_\omega(0)
+\frac{1}{\pi^2}\langle\,D^{1,0}\pr f\,,D^{1,0}\pr\ol g\,\rangle_\omega(0)\\
&\quad+\frac{1}{\pi^2}\langle\,D^{0,1}\ddbar f\,,D^{0,1}\ddbar\ol g\,\rangle_\omega(0).
\end{split}
\end{equation} 
From \eqref{s2-ecompVI}, \eqref{s2-ecompIV} and \eqref{s2-eXII-bis}, it is straightforward to check that
\begin{equation} \label{s2-ecompVII} 
\begin{split}
\triangle^2_\omega(fg)(0)&=(\triangle^2_\omega f)(0)g(0)+(\triangle^2_\omega g)(0)f(0)
+4\langle\,\ddbar f\wedge\pr g+\ddbar g\wedge\pr f\,,{\rm Ric\,}_\omega\,\rangle_\omega(0)\\
&\quad-4\langle\,\pr\triangle_\omega f\,,\pr\ol g\,\rangle_\omega(0)
-4\langle\,\ddbar\triangle_\omega f\,,\ddbar\ol g\,\rangle_\omega(0)
-4\langle\,\pr\triangle_\omega g\,,\pr\ol f\,\rangle_\omega(0)\\
&\quad-4\langle\,\ddbar\triangle_\omega g\,,\ddbar\,\ol f\,\rangle_\omega(0)+2(\triangle_\omega f)(0)(\triangle_\omega g)(0)
+8\langle\,\ddbar\pr f\,,\ddbar\pr\ol g\,\rangle_\omega(0)\\
&\quad+4\langle\,D^{1,0}\pr f\,,D^{1,0}\pr\ol g\,\rangle_\omega(0)+4\langle\,D^{0,1}\ddbar f\,,D^{0,1}\ddbar\ol g\,\rangle_\omega(0).
\end{split}
\end{equation}  
From \eqref{s2-ecompVII}, \eqref{s2-ecompII} follows.
\end{proof}

\section{The phase function version of the asymptotic expansion of the kernel of Berezin-Toeplitz quantization}

In this section, we will establish the phase function version of the asymptotic expansion of the kernel of Berezin-Toeplitz quantization which is important in our computations. 

We first review the phase function version of the asymptotic expansion of Bergman kernel.
As in section 1.3, let $\pit^{(k)}$ be the Bergman projection and let $\pit^{(k)}(x,y)\in C^\infty(X\times X,L^k_y\boxtimes L^k_x)$ be the distribution kernel of $\pit^{(k)}$. Let $s$ be a local trivializing section of $L$ over an open set $D\subset X$, $\abs{s}^2=e^{-2\phi}$.
Then on $D\times D$ we can write
\[\pit^{(k)}(x,y)=s(x)^k\pit^{(k)}_{s}(x, y)s^*(y)^k,\]
where $\pit^{(k)}_{s}(x, y)\in C^\infty(D\times D)$ so that for $x\in D$, $u\in C^\infty_0(D,L^k)$,
\begin{equation} \label{s3-eI}
\begin{split}
(\pit^{(k)}u)(x)&=s(x)^k\int_X\pit^{(k)}_{s}(x, y)<u(y), s^*(y)^k>dv_X(y)\\
&=s(x)^k\int_X\pit^{(k)}_{s}(x, y)\Td u(y)dv_X(y),\ \ u=s^k\Td u,\ \ \Td u\in C^\infty_0(D).
\end{split}
\end{equation} 

Let $F(x,y)\in C^\infty(D\times D)$. We say that $F$ is properly supported if ${\rm Supp\,}F\subset D\times D$ is proper. That is, the two projections: $t_x:(x,y)\in{\rm Supp\,}F\To x\in D$, $t_y:(x,y)\in{\rm Supp\,}F\To y\in D$ are proper (i.e. the inverse image of every compact subset of $D$ is compact). 

Catlin~\cite{Cat97} and Zelditch~\cite{Zel98} established the complete asymptotic expansion for $\pit^{(k)}_s$
on the diagonal by using a result of Boutet de Monvel-Sj\"{o}strand~\cite{BouSj76} for the asymptotics of the Szeg\"{o} kernel on a strictly pseudoconvex boundary, here on the boundary of the unit disc bundle and a reduction 
idea of Boutet de Monvel-Guillemin~\cite{BouGu81}. Dai-Lui-Ma~\cite{DLM04a}, Berman-Berndtsson-Sj\"{o}strand~\cite{BBS04} and Ma-Marinescu~\cite{MM08a} obtained the full off-diagonal expansion for $e^{-k\phi(z)+k\phi(w)}\pit^{(k)}_s(z,w)$ by using different methods. Hsiao-Marinescu~\cite{HM11} established the full off-diagonal expansion for the Bergman kernel for lower energy forms without the assumption that $L$ is positive. 
When $L$ is positive, we deduce the full off-diagonal expansion for the Bergman kernel. More precisely, 
we have the following

\begin{thm} \label{s3-tI} 
We recall that we work with Assumption~\ref{s1-a1}. With the notations used above let $D\subset X$ be an open set with holomorphic coordinates $z=(z_1,\ldots,z_n)$ and let $s$ be a local trivializing section of $L$ on 
$D\subset X$ and $\abs{s}^2=e^{-2\phi}$. We also write $w=(w_1,\ldots,w_n)$. Then, 
\begin{equation} \label{s3-eII} 
e^{-k\phi(z)+k\phi(w)}\pit^{(k)}_s(z,w)\equiv e^{ik\Psi(z,w)}b(z,w,k)\mod O(k^{-\infty})\ \ \mbox{on $D\times D$},
\end{equation}
where $b(z,w,k)\in C^\infty(D\times D)$ is properly supported and 
\begin{equation} \label{s3-eIII}
\begin{split}
&b(z,w,k)\equiv\sum^{\infty}_{j=0}k^{n-j}b_j(z,w)\mod O(k^{-\infty})\ \ \mbox{on $D\times D$},\\
&b_j(z, w)\in C^\infty(D\times D),\ \ j=0,1,2,\ldots,\\ 
&\mbox{$\ddbar_zb_j(z,w)$ and $\pr_wb_j(z,w)$ vanish to infinite order at $z=w$, for all $j=0,1,\ldots$,}
\end{split}
\end{equation}
$\Psi(z,w)\in C^\infty(D\times D)$,
$\Psi(z,w)=-\ol\Psi(w,z)$, ${\rm Im\,}\Psi\geq c\abs{z-w}^2$, $c>0$, $\Psi=0$ if amd only if $z=w$. 
Moreover, for a given point $p\in D$, if we take local holomorphic coordinates
$z=(z_1,\ldots,z_n)$ vanishing at $p$, then we have 
\begin{equation} \label{s3-eIV} 
\Psi(z,w)=i\bigr(\phi(z)+\phi(w)\bigr)-2i\sum_{\alpha,\beta\in\mathbb N_0,\abs{\alpha}+\abs{\beta}\leq N}\frac{\pr^{\abs{\alpha}+\abs{\beta}}\phi}{\pr z^\alpha\pr\ol z^\beta}(0)\frac{z^\alpha}{\alpha!}\frac{\ol w^\beta}{\beta!}+O(\abs{(z,w)}^{N+1}),
\end{equation} 
for every $N\in\mathbb N_0$. 
\end{thm}

From \eqref{s1-e3-1} and Theorem~\ref{s3-tI}, we deduce:

\begin{cor} \label{s3-cI}
With the notations used above let $D\subset X$ be an open set with holomorphic coordinates $z=(z_1,\ldots,z_n)$ and let $s$ be a local trivializing section of $L$ on 
$D\subset X$. We also write $w=(w_1,\ldots,w_n)$, $u=(u_1,\ldots,u_n)$. Let $f\in C^\infty(X)$. Then, we have 
\begin{equation}\label{s3-eVI}
e^{-k\phi(z)+k\phi(w)}T^{(k)}_{f,s}(z,w)\equiv\int_D e^{ik(\Psi(z,u)+\Psi(u,w))}b(z,u,k)f(u)b(u,w,k)dV_X(u)\mod O(k^{-\infty})
\end{equation}
on $D$, where $b(z,w,k)\in C^\infty(D\times D)$ and $\Psi(z,w)\in C^\infty(D\times D)$ are as in Theorem~\ref{s3-tI}.

In particular, we have
\begin{equation} \label{s3-eV}
T^{(k)}_f(z)\equiv\int_D e^{ik(\Psi(z,u)+\Psi(u,z))}b(z,u,k)f(u)b(u,z,k)dV_X(u)\mod O(k^{-\infty})\ \ \mbox{on $D$}.
\end{equation}
We remind that the kernel $T^{(k)}_{f,s}(z,w)$ is given by \eqref{s1-e4}.
\end{cor}

Now, we study the kernel $e^{-k\phi(z)+k\phi(w)}T^{(k)}_{f,s}(z,w)$. Until further notice, we work on $D$. We write $z=x=(x_1,\ldots,x_{2n})$, $z_j=x_{2j-1}+ix_{2j}$, 
$j=1,\ldots,n$, $w=y=(y_1,\ldots,y_{2n})$, $w_j=y_{2j-1}+iy_{2j}$, $j=1,\ldots,n$, $u=\alpha=(\alpha_1,\ldots,\alpha_{2n})$, $u_j=\alpha_{2j-1}+i\alpha_{2j}$, $j=1,\ldots,n$. Put 
\[F(x,\alpha,y)=:i\Psi(x,\alpha)+i\Psi(\alpha,y).\] 
Since $\Psi(x,\alpha)=-\ol\Psi(\alpha,x)$, we have $F(x,\alpha,x)=-2{\rm Im\,}\Psi(x,\alpha)$. Note that ${\rm Im\,}\Psi(x,\alpha)\geq0$ and ${\rm Im\,}\Psi(x,x)=0$. From this observation, we can check that 
\[\rabs{d_\alpha F(x,\alpha,y)}_{x=y=\alpha}=-2\rabs{{\rm Im\,}d_\alpha\Psi(x,\alpha)}_{x=\alpha}=0.\] 
Moreover, from \eqref{s3-eIV}, it is not difficult to check that 
\[\rabs{\det\left(\frac{\pr^2F}{\pr\alpha_j\pr\alpha_k}\right)^{2n}_{j,k=1}}_{x=y=\alpha}=2^{2n}\bigr(\det\dot R^L(x)\bigr)^2.\] 
Thus, $x=y=\alpha$ are real critical points and $F(x,\alpha,y)$ is a non-degenerate complex valued phase function in 
the sense of Melin-Sj\"{o}strand~\cite{MS74}. 
We can apply the stationary phase formula of Melin-Sj\"{o}strand~\cite{MS74} to carry out the integral in \eqref{s3-eVI} and obtain 
\begin{equation}\label{s3-eVII}
e^{-k\phi(x)+k\phi(y)}T^{(k)}_{f,s}(x,y)\equiv e^{ik\hat\Psi(x,y)}\hat b_f(x,y,k)\mod O(k^{-\infty})\ \ \mbox{on $D\times D$},
\end{equation}
where $\hat b_f(x,y,k)\in C^\infty(D\times D)$ is properly supported,
\begin{equation} \label{s3-eVIII}
\begin{split}
&\hat b_f(x,y,k)\equiv\sum^{\infty}_{j=0}k^{n-j}\hat b_{j,f}(x,y)\mod O(k^{-\infty})\ \ \mbox{on $D\times D$},\\
&\hat b_{j,f}(x, y)\in C^\infty(D\times D),\ \ j=0,1,2,\ldots,
\end{split}
\end{equation}
and $\hat\Psi(x,y)\in C^\infty(D\times D)$, ${\rm Im\,}\hat\Psi\geq0$, $\hat\Psi(x,x)=0$. We claim that 
\begin{equation} \label{s3-eIX}
\mbox{$\Psi(x,y)-\hat\Psi(x,y)$ vanishes to infinite order at $x=y$}.
\end{equation}
Let $f=1$ and notice that the phase $\hat\Psi$ is independent of $f$ and 
\begin{equation} \label{s3-eX}
e^{-k\phi(x)+k\phi(y)}T^{(k)}_{1,s}(x,y)
=e^{-k\phi(x)+k\phi(y)}\pit^{(k)}_{s}(x,y).
\end{equation}
From this observation and \eqref{s3-eII}, we conclude that 
\begin{equation}\label{s3-eXbis}
e^{ik\hat\Psi(x,y)}\hat b_1(x,y,k)=e^{ik\Psi(x,y)}b(x,y,k)+G_k(x,y),
\end{equation}
where $G_k(x,y)\equiv0\mod O(k^{-\infty})$. That is, for every $N\in\mathbb N_0$, every $\alpha,\beta\in\mathbb N_0^{2n}$ and every compact set $K\subset D$, there exists a constant $C_{N,\alpha,K}>0$ independent of $k$, such that 
\[\abs{\pr^\alpha_x\pr^\beta_yG_k(x,y)}\leq C_{N,\alpha,\beta,K}k^{-N},\] 
$x, y\in K$, for $k$ large. We assume that there exist $\alpha_0, \beta_0\in\mathbb N^{2n}_0$ and $(x_0,x_0)\in D\times D$, such that 
\[\rabs{\pr^{\alpha_0}_x\pr^{\beta_0}_y(i\Psi(x,y)-i\hat\Psi(x,y))}_{(x_0,x_0)}=C_{\alpha_0,\beta_0}\neq0,\]
and
\[\rabs{\pr^{\alpha}_x\pr^{\beta}_y(i\Psi(x,y)-i\hat\Psi(x,y))}_{(x_0,x_0)}=0\ \ \mbox{if $\abs{\alpha}+\abs{\beta}<\abs{\alpha_0}+\abs{\beta_0}$}.\] 
From \eqref{s3-eXbis}, we have 
\begin{equation}\label{s3-eXI}
\begin{split}
&\rabs{\pr^{\alpha_0}_x\pr^{\beta_0}_y\Bigr(e^{ik\Psi(x,y)-ik\hat\Psi(x,y)}b(x,y,k)-\hat b_1(x,y,k)\Bigr)}_{(x_0,x_0)}\\
&=-\rabs{\pr^{\alpha_0}_x\pr^{\beta_0}_y\Bigr(e^{-ik\hat\Psi(x,y)}G_k(x,y)\Bigr)}_{(x_0,x_0)}.
\end{split}
\end{equation}
Since $\hat\Psi(x_0,x_0)=0$, we have 
\begin{equation} \label{s3-eXII}
\lim_{k\To\infty}k^{-n-1}\rabs{\pr^{\alpha_0}_x\pr^{\beta_0}_y\Bigr(e^{-ik\hat\Psi(x,y)}G_k(x,y)\Bigr)}_{(x_0,x_0)}=0.
\end{equation} 
On the other hand, we can check that 
\begin{equation} \label{s3-eXIII}
\lim_{k\To\infty}k^{-n-1}\rabs{\pr^{\alpha_0}_x\pr^{\beta_0}_y\Bigr(e^{ik\Psi(x,y)-ik\hat\Psi(x,y)}b(x,y,k)-\hat b_1(x,y,k)\Bigr)}_{(x_0,x_0)}=C_{\alpha_0,\beta_0}b_0(x_0,x_0)\neq0
\end{equation} 
since $b_0(x_0,x_0)\neq0$. From \eqref{s3-eXII}, \eqref{s3-eXIII} and \eqref{s3-eXI}, we get a contradiction. 
The claim \eqref{s3-eIX} follows. 

From \eqref{s3-eIX} and the global theory of Fourier integral operators~\cite{MS74}, we can replace the phase $\hat\Psi$ 
by $\Psi$. Thus, 
\begin{equation}\label{s3-eXIV}
e^{-k\phi(x)+k\phi(y)}T^{(k)}_{f,s}(x,y)\equiv e^{ik\Psi(x,y)}b_f(x,y,k)\mod O(k^{-\infty})\ \ \mbox{on $D\times D$},
\end{equation}
where $b_f(x,y,k)\in C^\infty(D\times D)$ is properly supported,
\begin{equation} \label{s3-eXV}
\begin{split}
&b_f(x,y,k)\equiv\sum^{\infty}_{j=0}k^{n-j}b_{j,f}(x,y)\mod O(k^{-\infty})\ \ \mbox{on $D\times D$},\\
&b_{j,f}(x, y)\in C^\infty(D\times D),\ \ j=0,1,2,\ldots.
\end{split}
\end{equation} 

We claim that 
\begin{equation} \label{s3-eXVI}
\mbox{$\ddbar_zb_{j,f}(x,y)$ and $\pr_wb_{j,f}(x,y)$ vanish to infinite order at $x=y$, for all $j=0,1,\ldots$.}
\end{equation}
In view of \eqref{s3-eIV}, we see that $\ddbar_z(i\Psi(x,y)+\phi(x))$ vanishes to infinite order at $x=y$. From this observation and notice that $\ddbar_zT_{f,s}(x,y)=0$, we conclude that 
\begin{equation}\label{s3-eXVII}
e^{ik\Psi(x,y)}\ddbar_zb_f(x,y,k)=H_k(x,y),
\end{equation}
where $H_k(x,y)\equiv0\mod O(k^{-\infty})$. We assume that there exist $\gamma_0, \delta_0\in\mathbb N^{2n}_0$ and $(x_1,x_1)\in D\times D$, such that 
\[\rabs{\pr^{\gamma_0}_x\pr^{\delta_0}_y(\ddbar_zb_{0,f}(x,y))}_{(x_1,x_1)}=D_{\gamma_0,\delta_0}\neq0,\]
and
\[\rabs{\pr^{\gamma}_x\pr^{\delta}_y(\ddbar_zb_{0,f}(x,y))}_{(x_1,x_1)}=0\ \ \mbox{if $\abs{\gamma}+\abs{\delta}<\abs{\gamma_0}+\abs{\delta_0}$}.\] 
From \eqref{s3-eXVII}, we have 
\begin{equation}\label{s3-eXVIII}
\rabs{\pr^{\gamma_0}_x\pr^{\delta_0}_y\Bigr(\ddbar_zb_f(x,y,k)\Bigr)}_{(x_1,x_1)}
=\rabs{\pr^{\gamma_0}_x\pr^{\delta_0}_y\Bigr(e^{-ik\Psi(x,y)}H_k(x,y)\Bigr)}_{(x_1,x_1)}.
\end{equation}
Since $\Psi(x_1,x_1)=0$, we have 
\begin{equation} \label{s3-eXIX}
\lim_{k\To\infty}k^{-n}\rabs{\pr^{\gamma_0}_x\pr^{\delta_0}_y\Bigr(e^{-ik\Psi(x,y)}H_k(x,y)\Bigr)}_{(x_1,x_1)}=0.
\end{equation} 
On the other hand, we can check that 
\begin{equation} \label{s3-eXX}
\lim_{k\To\infty}k^{-n}\rabs{\pr^{\gamma_0}_x\pr^{\delta_0}_y\Bigr(\ddbar_zb_f(x,y,k)\Bigr)}_{(x_1,x_1)}=D_{\gamma_0,\delta_0}\neq0.
\end{equation} 
From \eqref{s3-eXX}, \eqref{s3-eXIX} and \eqref{s3-eXVIII}, we get a contradiction. 
Thus, $\ddbar_zb_{0,f}(x,y)$ vanishes to infinite order at $x=y$. Similarly, we can repeat the procedure above and conclude 
that $\ddbar_zb_{j,f}(x,y)$ and $\pr_wb_{j,f}(x,y)$ vanish to infinite order at $x=y$, $j=0,1,\ldots$. The claim \eqref{s3-eXVI} follows.  

Summing up, we obtain the phase function version of the asymptotic expansion of the kernel of Berezin-Toeplitz quantization 

\begin{thm} \label{s3-tII}
We recall that we work with Assumption~\ref{s1-a1}. With the notations used above let $D\subset X$ be an open set with holomorphic coordinates $z=(z_1,\ldots,z_n)$ and let $s$ be a local trivializing section of $L$ on 
$D\subset X$ and $\abs{s}^2=e^{-2\phi}$. We also write $w=(w_1,\ldots,w_n)$. Let $f\in C^\infty(X)$ and let $T^{(k)}_{f,s}(z,w)$ be as in \eqref{s1-e4}. Then, 
\begin{equation} \label{s3-eII-bisbbbbb} 
e^{-k\phi(z)+k\phi(w)}T^{(k)}_{f,s}(z,w)\equiv e^{ik\Psi(z,w)}b_f(z,w,k)\mod O(k^{-\infty})\ \ \mbox{on $D\times D$},
\end{equation}
where $b_f(z,w,k)\in C^\infty(D\times D)$ is properly supported and 
\begin{equation} \label{s3-eXXI}
\begin{split}
&b_f(z,w,k)\equiv\sum^{\infty}_{j=0}k^{n-j}b_{j,f}(z,w)\mod O(k^{-\infty})\ \ \mbox{on $D\times D$},\\
&b_{j,f}(z, w)\in C^\infty(D\times D),\ \ j=0,1,2,\ldots,\\ 
&\mbox{$\ddbar_zb_{j,f}(z,w)$ and $\pr_wb_{j,f}(z,w)$ vanish to infinite order at $z=w$, for all $j=0,1,\ldots$,}
\end{split}
\end{equation} 
and $\Psi(z,w)\in C^\infty(D\times D)$ is as in Theorem~\ref{s3-tI}.
\end{thm} 

From Theorem~\ref{s3-tII}, we deduce:

\begin{cor} \label{s3-cII}
With the notations used above let $D\subset X$ be an open set with holomorphic coordinates $z=(z_1,\ldots,z_n)$ and let $s$ be a local trivializing section of $L$ on 
$D\subset X$. We also write $w=(w_1,\ldots,w_n)$. Let $f, g\in C^\infty(X)$. Then, we have 
\begin{equation}\label{s3-eVI-bisbbb}
(T^{(k)}_f\circ T^{(k)}_g)(z)\equiv\int_D e^{ik(\Psi(z,w)+\Psi(w,z))}b_f(z,w,k)b_g(w,z,k)dV_X(w)\mod O(k^{-\infty})
\end{equation}
on $D$, where $b_f(z,w,k), b_g(w,z,k)\in C^\infty(D\times D)$ are as in Theorem~\ref{s3-tII} and $\Psi(z,w)\in C^\infty(D\times D)$ is as in Theorem~\ref{s3-tI}.
\end{cor}

\section{The coefficients of the asymptotic expansion of the kernel of Berezin-Toeplitz quantization} 

Let $f\in C^\infty(X)$ and let $b_{j,f}$, $j=0,1,\ldots$, be as in \eqref{s1-e5}. Fix a point $p\in X$. In this section, 
we will give a method for computing $b_{j,f}(p)$, $j=0,1,\ldots$, and we will compute the first three terms explicitly. 
Near $p$, we take local holomorphic coordinates $z=(z_1,\ldots,z_n)$, $z_j=x_{2j-1}+ix_{2j}$, $j=1,\ldots,n$, and local section $s$ defined in some small neighborhood $D$ of $p$ such that \eqref{s2-eI} holds. Until further notice, we work 
with this local coordinates $z$ and we identify $p$ with the point $x=z=0$. 

In view of \eqref{s3-eV}, we see that 
\begin{equation} \label{s4-eI}
T^{(k)}_f(0)=\int_De^{ik(\Psi(0,z)+\Psi(z,0))}b(0,z,k)b(z,0,k)f(z)V_\Theta(z)d\lambda(z)+r_k,
\end{equation}
where $d\lambda(z)=2^ndx_1dx_2\cdots dx_{2n}$, $V_\Theta$ is given by \eqref{s1-e10} and 
\[\lim_{k\to\infty}\frac{r_k}{k^N}=0,\ \ \forall N\geq0.\] 
We notice that since $b(z,w,k)$ is properly supported, we have 
\begin{equation} \label{s4-eII}
b(0,z,k)\in C^\infty_0(D),\ \ b(z,0,k)\in C^\infty_0(D).
\end{equation}
We recall the stationary phase formula of H\"{o}rmander (see Theorem~7.7.5 in~\cite{Hor03}) 

\begin{thm} \label{s4-tI} 
Let $K\subset D$ be a compact set and $N$ a positive integer. If $u\in C^\infty_0(K)$, 
$F\in C^\infty(D)$ and ${\rm Im\,}F\geq0$ in $D$, ${\rm Im\,}F(0)=0$, $F'(0)=0$, ${\rm det\,}F''(0)\neq0$, $F'\neq0$ in $K\setminus\set{0}$ then 
\begin{equation} \label{s4-eIII} 
\begin{split}
&\abs{\int e^{ikF(z)}u(z)V_\Theta(z)d\lambda(z)-2^ne^{ikF(0)}{\rm det\,}\left(\frac{kF''(0)}{2\pi i}\right)^{-\frac{1}{2}}\sum_{j<N}k^{-j}L_ju} \\
&\quad\leq Ck^{-N}\sum_{\abs{\alpha}\leq 2N}\sup{\abs{\pr^\alpha_x u}},\ \ k>0,
\end{split}
\end{equation} 
where $C$ is bounded when $F$ stays in a bounded set in $C^\infty(D)$ and $\frac{\abs{x}}{\abs{F'(x)}}$ has a uniform bounded and
\begin{equation} \label{s4-eIV} 
L_ju=\sum_{\nu-\mu=j}\sum_{2\nu\geq 3\mu}i^{-j}2^{-\nu}<F''(0)^{-1}D,D>^\nu\frac{(h^\mu V_\Theta u)(0)}{\nu!\mu!}.
\end{equation} 
Here 
\begin{equation} \label{s4-eV} 
h(x)=F(x)-F(0)-\frac{1}{2}<F''(0)x,x>
\end{equation} 
and $D=\left(\begin{array}[c]{ccc}
  &-i\pr_{x_1}  \\
  &\vdots\\
  &-i\pr_{x_{2n}}
\end{array}\right)$.
\end{thm} 

Now, we apply \eqref{s4-eIII} to the integral in \eqref{s4-eI}. Put 
\[F(z)=\Psi(0,z)+\Psi(z,0).\] 
From \eqref{s3-eIV} and \eqref{s2-eI}, we see that 
\begin{equation} \label{s4-eVI}
\begin{split}
&F(z)=2i\sum^n_{j=1}\lambda_j\abs{z_j}^2+2i\phi_1(z)+O(\abs{z}^N),\ \ \forall N\geq0,\\
&h(z)=2i\phi_1(z)+O(\abs{z}^N),\ \ \forall N\geq0,
\end{split}
\end{equation}
where $h$ is given by \eqref{s4-eV}. Moreover, we can check that 
\begin{equation} \label{s4-eVII}
{\rm det\,}\left(\frac{kF''(0)}{2\pi i}\right)^{-\frac{1}{2}}=k^{-n}\pi^n2^{-n}\lambda^{-1}_1\lambda^{-1}_2\cdots\lambda_n^{-1}=k^{-n}\pi^n\bigr(\det\dot{R}^L(0)\bigr)^{-1}
\end{equation}
and
\begin{equation} \label{s4-eVIII}
<F''(0)^{-1}D,D>=i\triangle_0.
\end{equation}
We recall that $\triangle_0$ is given by \eqref{s2-eV-I}. From \eqref{s4-eVI}, \eqref{s4-eVIII} and notice that $h=O(\abs{z}^4)$, it is not 
difficult to see that  
\begin{equation} \label{s4-eIX} 
L_j(b(0,z,k)b(z,0,k)f)=\sum_{\nu-\mu=j}\sum_{2\nu\geq 4\mu}(-1)^\mu2^{-j}\frac{\triangle_0^\nu\bigr(\phi_1^\mu V_\Theta b(0,z,k)b(z,0,k)f\bigr)(0)}{\nu!\mu!},
\end{equation} 
where $L_j$ is given by \eqref{s4-eIV}. We notice that 
\[b(0,z,k)\equiv\sum^\infty_{j=0}b_j(0,z)k^{n-j}\mod O(k^{-\infty}),\ \ 
b(z,0,k)\equiv\sum^\infty_{j=0}b_j(z,0)k^{n-j}\mod O(k^{-\infty}).\] 
From this observation, \eqref{s4-eIX} becomes: 
\begin{equation} \label{s4-eX} 
\begin{split}
&L_j(b(0,z,k)b(z,0,k)f)\\
&=\sum_{\nu-\mu=j}\sum_{2\nu\geq 4\mu}\sum_{0\leq s+t\leq N}(-1)^\mu2^{-j}
\frac{k^{2n-s-t}\triangle_0^\nu\bigr(\phi_1^\mu V_\Theta fb_s(0,z)b_t(z,0)\bigr)(0)}{\nu!\mu!}+O(k^{2n-N-1}),
\end{split}
\end{equation} 
for all $N\geq0$. From \eqref{s4-eX}, \eqref{s4-eVII}, \eqref{s4-eIII} and \eqref{s4-eI}, we get 
\begin{equation} \label{s4-eXI}
\begin{split}
&T^{(k)}_f(0)=(2\pi)^n(\det\dot{R}^L(0))^{-1}\times\\
&\sum^N_{j=0}k^{n-j}\Bigr(\sum_{0\leq m\leq j}\sum_{\nu-\mu=m}\sum_{2\nu\geq 4\mu}\sum_{s+t=j-m}(-1)^\mu2^{-m}
\frac{\triangle_0^\nu\bigr(\phi_1^\mu V_\Theta fb_s(0,z)b_t(z,0)\bigr)(0)}{\nu!\mu!}\Bigr)\\
&\quad+O(k^{n-N-1}),\ \ \forall N\geq0.
\end{split}
\end{equation} 
Combining \eqref{s4-eXI} with \eqref{s1-e5}, we obtain

\begin{thm} \label{s4-tII} 
For $b_{j,f}$, $j=0,1,\ldots$, in \eqref{s1-e5}, we have 
\begin{equation} \label{s4-eXII}
\begin{split}
&b_{j,f}(0)\\
&=(2\pi)^n(\det\dot{R}^L(0))^{-1}\sum_{0\leq m\leq j}\sum_{\nu-\mu=m}\sum_{2\nu\geq 4\mu}\sum_{s+t=j-m}(-1)^\mu2^{-m}
\frac{\triangle_0^\nu\bigr(\phi_1^\mu V_\Theta fb_s(0,z)b_t(z,0)\bigr)(0)}{\nu!\mu!},
\end{split}
\end{equation} 
for all $j=0,1,\ldots$. 

In particular, 
\begin{equation} \label{s4-eXIII}
b_{0,f}(0)=(2\pi)^n(\det\dot{R}^L(0))^{-1}f(0)b_0(0,0)^2,
\end{equation}
\begin{equation} \label{s4-eXIV}
\begin{split}
b_{1,f}(0)&=(2\pi)^n(\det\dot{R}^L(0))^{-1}\Bigr(2f(0)b_0(0,0)b_1(0,0)\\
&\quad+\frac{1}{2}\triangle_0\bigr(V_\Theta fb_0(0,z)b_0(z,0)\bigr)(0)
-\frac{1}{4}\triangle^2_0\bigr(\phi_1V_\Theta fb_0(0,z)b_0(z,0)\bigr)(0)\Bigr)
\end{split}
\end{equation}
and 
\begin{equation} \label{s4-eXV}
\begin{split}
b_{2,f}(0)&=(2\pi)^n(\det\dot{R}^L(0))^{-1}\Bigr(2f(0)b_0(0,0)b_2(0,0)+f(0)b_1(0,0)^2\\
&\quad+\frac{1}{2}\triangle_0\bigr(V_\Theta f(b_0(0,z)b_1(z,0)+b_1(0,z)b_0(z,0))\bigr)(0)\\
&\quad-\frac{1}{4}\triangle^2_0\bigr(\phi_1V_\Theta f(b_0(0,z)b_1(z,0)+b_1(0,z)b_0(z,0))\bigr)(0)\\
&\quad+\frac{1}{8}\triangle^2_0\bigr(V_\Theta fb_0(0,z)b_0(z,0)\bigr)(0)
-\frac{1}{24}\triangle^3_0\bigr(\phi_1V_\Theta fb_0(0,z)b_0(z,0)\bigr)(0)\\
&\quad+\frac{1}{192}\triangle^4_0\bigr(\phi_1^2V_\Theta fb_0(0,z)b_0(z,0)\bigr)(0)\Bigr).
\end{split}
\end{equation}
\end{thm}

\subsection{The coefficient $b_{0,f}$}
We notice that when $f=1$, $b_{0,f}(0)=b_0(0,0)$. From this observation and \eqref{s4-eXIII}, we obtain
\[b_0(0,0)=(2\pi)^n(\det\dot{R}^L(0))^{-1}b_0(0,0)^2.\]
Thus,
\begin{equation} \label{s4-eXVI}
b_0(0,0)=(2\pi)^{-n}\det\dot{R}^L(0).
\end{equation}
Combining \eqref{s4-eXVI} with \eqref{s4-eXIII}, we get
\begin{equation}\label{s4-eXVII}
b_{0,f}(0)=(2\pi)^{-n}\det\dot{R}^L(0)f(0).
\end{equation}
From this, \eqref{s1-e16-I} follows. 

\subsection{The coefficient $b_{1,f}$}

In view of \eqref{s4-eXIV}, we see that to compute $b_{1,f}(0)$ we have to know which global geometric functions
at $z=0$ equal to $\triangle_0\bigr(V_\Theta fb_0(0,z)b_0(z,0)\bigr)(0)$ and $\triangle_0^2\bigr(\phi_1V_\Theta fb_0(0,z)b_0(z,0)\bigr)(0)$.
Now we compute $\triangle_0\bigr(V_\Theta fb_0(0,z)b_0(z,0)\bigr)(0)$. First, we need

\begin{lem}\label{s4-lI}
We have
\begin{equation}\label{s4-eXVIII}\begin{split}
\rabs{\frac{\pr b_0(z,0)}{\pr z_s}}_{z=0}&=-(2\pi)^{-n}\det\dot R^L(0)\frac{\pr V_\Theta}{\pr z_s}(0),\\
\rabs{\frac{\pr b_0(0,z)}{\pr\ol z_s}}_{z=0}&=-(2\pi)^{-n}\det\dot R^L(0)\frac{\pr V_\Theta}{\pr\ol z_s}(0),
\end{split}\qquad s=1,\ldots,n,
\end{equation}
\begin{equation}\label{s4-eXIX}\begin{split}
\rabs{\frac{\pr^2 b_0(z,0)}{\pr z_s\pr z_t}}_{z=0}&=(2\pi)^{-n}\det\dot R^L(0)\bigr(2\frac{\pr V_\Theta}{\pr z_s}(0)
\frac{\pr V_\Theta}{\pr z_t}(0)-\frac{\pr^2 V_\Theta}{\pr z_s\pr z_t}(0)\bigr),\\
\rabs{\frac{\pr^2 b_0(0,z)}{\pr\ol z_s\pr\ol z_t}}_{z=0}&=(2\pi)^{-n}\det\dot R^L(0)\bigr(2\frac{\pr V_\Theta}{\pr\ol z_s}(0)
\frac{\pr V_\Theta}{\pr\ol z_t}(0)-\frac{\pr^2 V_\Theta}{\pr\ol z_s\pr\ol z_t}(0)\bigr),
\end{split}\qquad s,t=1,\ldots,n,
\end{equation}
\begin{equation}\label{s4-eXX}\begin{split}
&\rabs{\frac{\pr^{\abs{\alpha}+\abs{\beta}}b_0(z,0)}{\pr z^\alpha\pr\ol z^\beta}}_{z=0}=0,\qquad
\forall\alpha,\beta\in\mathbb N^n_0,\ \ \abs{\beta}>0,\\
&\rabs{\frac{\pr^{\abs{\gamma}+\abs{\delta}}b_0(0,z)}{\pr z^\gamma\pr\ol z^\delta}}_{z=0}=0,\qquad
\forall\gamma,\delta\in\mathbb N^n_0,\ \ \abs{\gamma}>0.
\end{split}
\end{equation}
\end{lem}

\begin{proof}
In view of \eqref{s3-eIII}, we see that 
\begin{equation}\label{s4-eXXI}
\ddbar_zb_0(z,0)=O(\abs{z}^N),\qquad \pr_zb_0(0,z)=O(\abs{z}^N),\qquad\forall N\geq 0.
\end{equation}
From \eqref{s4-eXXI}, \eqref{s4-eXX} follows.

From \eqref{s4-eXVI}, we see that
\begin{equation}\label{s4-eXXII}
b_0(z,z)=(2\pi)^{-n}\det\dot{R}^L(z)=V_\omega(z)\bigr(V_\Theta(z)\bigr)^{-1}.
\end{equation}
In view of \eqref{s2-eI}, we see that
\begin{equation}\label{s4-eXXIII}
\frac{\pr^{\abs{\alpha}}V_\omega}{\pr\ol z^\alpha}(0)=\frac{\pr^{\abs{\alpha}}V_\omega}{\pr z^\alpha}(0)=0,
\qquad\forall\alpha\in\mathbb N_0^n.
\end{equation}
From \eqref{s4-eXXII}, \eqref{s4-eXXI} and \eqref{s4-eXXIII}, we have
\begin{equation}\label{s4-eXXIV}\begin{split}
\rabs{\frac{\pr^{\abs{\alpha}}b_0(z,z)}{\pr z^\alpha}}_{z=0}
&=\rabs{\frac{\pr^{\abs{\alpha}}b_0(z,0)}{\pr z^\alpha}}_{z=0}
=V_\omega(0)\rabs{\frac{\pr^{\abs{\alpha}}}{\pr z^\alpha}\bigr(V_\Theta(z)\bigr)^{-1}}_{z=0}\\
&=(2\pi)^{-n}\det\dot R^L(0)\rabs{\frac{\pr^{\abs{\alpha}}}{\pr z^\alpha}\bigr(V_\Theta(z)\bigr)^{-1}}_{z=0},\\
\rabs{\frac{\pr^{\abs{\alpha}}b_0(z,z)}{\pr\ol z^\alpha}}_{z=0}
&=\rabs{\frac{\pr^{\abs{\alpha}}b_0(0,z)}{\pr\ol z^\alpha}}_{z=0}
=V_\omega(0)\rabs{\frac{\pr^{\abs{\alpha}}}{\pr\ol z^\alpha}\bigr(V_\Theta(z)\bigr)^{-1}}_{z=0}\\
&=(2\pi)^{-n}\det\dot R^L(0)\rabs{\frac{\pr^{\abs{\alpha}}}{\pr\ol z^\alpha}\bigr(V_\Theta(z)\bigr)^{-1}}_{z=0},
\end{split}
\end{equation}
for all $\alpha\in\mathbb N_0^n$. We can compute
\begin{equation}\label{s4-eXXV}
\rabs{\frac{\pr}{\pr z_s}\bigr(V_\Theta(z)\bigr)^{-1}}_{z=0}=-\frac{\pr V_\Theta}{\pr z_s}(0),\ \ 
\rabs{\frac{\pr}{\pr\ol z_s}\bigr(V_\Theta(z)\bigr)^{-1}}_{z=0}=-\frac{\pr V_\Theta}{\pr\ol z_s}(0),\ \ 
s=1,\ldots,n,
\end{equation}
\begin{equation}\label{s4-eXXVI}\begin{split}
\rabs{\frac{\pr^2}{\pr z_s\pr z_t}\bigr(V_\Theta(z)\bigr)^{-1}}_{z=0}& =2\frac{\pr V_\Theta}{\pr z_s}(0)\frac{\pr V_\Theta}{\pr z_t}(0)
-\frac{\pr^2V_\Theta}{\pr z_s\pr z_t}(0),\\ 
\rabs{\frac{\pr^2}{\pr\ol z_s\pr\ol z_t}\bigr(V_\Theta(z)\bigr)^{-1}}_{z=0}& =2\frac{\pr V_\Theta}{\pr\ol z_s}(0)
\frac{\pr V_\Theta}{\pr\ol z_t}(0)-\frac{\pr^2V_\Theta}{\pr\ol z_s\pr\ol z_t}(0),
\end{split}\qquad s,t=1,\ldots,n. 
\end{equation}
From \eqref{s4-eXXV}, \eqref{s4-eXXVI} and \eqref{s4-eXXIV}, we obtain \eqref{s4-eXVIII} and \eqref{s4-eXIX}.
\end{proof}

From \eqref{s4-eXVIII} and \eqref{s4-eXX}, it is straightforward to see that
\begin{equation}\label{s4-eXXVII}\begin{split}
\triangle_0\bigr(V_\Theta fb_0(0,z)b_0(z,0)\bigr)(0)& =\sum^n_{j=1}\frac{1}{\lambda_j}\Bigr(\frac{\pr^2 V_\Theta}{\pr z_j\pr\ol z_j}(0)
-\abs{\frac{\pr V_\Theta}{\pr z_j}(0)}^2\Bigr)b_0(0,0)^2f(0)\\
&\quad+\bigr(\triangle_0f\bigr)(0)b_0(0,0)^2.\end{split}
\end{equation}
Combining \eqref{s4-eXXVII} with \eqref{s2-eX}, \eqref{s2-eXII-I} and \eqref{s4-eXVI}, we get
\begin{equation}\label{s4-eXXVIII}
\triangle_0\bigr(V_\Theta fb_0(0,z)b_0(z,0)\bigr)(0)=\Bigr(-\frac{1}{2\pi}\hat r(0)f(0)-\frac{1}{2\pi}\bigr(\triangle_\omega f\bigr)(0)\Bigr)
(2\pi)^{-2n}(\det\dot R^L(0))^2.
\end{equation} 

Now, we compute $\triangle^2_0\bigr(\phi_1V_\Theta fb_0(0,z)b_0(z,0)\bigr)(0)$. Since $\phi_1=O(\abs{z}^4)$, we have
\begin{equation}\label{s4-eXXIX}\begin{split}
\triangle^2_0\bigr(\phi_1V_\Theta fb_0(0,z)b_0(z,0)\bigr)(0)& =\bigr(\triangle^2_0\phi_1\bigr)(0)f(0)b_0(0,0)^2 \\
&=\bigr(\triangle^2_0\phi\bigr)(0)f(0)b_0(0,0)^2 \\
&=\sum^n_{s,t=1}\frac{1}{\lambda_s\lambda_t}\frac{\pr^4\phi}{\pr\ol z_s\pr z_s\pr\ol z_t\pr z_t}(0)f(0)b_0(0,0)^2.
\end{split}
\end{equation}
Combining \eqref{s4-eXXIX} with \eqref{s2-eIX} and \eqref{s4-eXVI}, we get
\begin{equation}\label{s4-eXXX}
\triangle^2_0\bigr(\phi_1V_\Theta fb_0(0,z)b_0(z,0)\bigr)(0)=-\frac{1}{2\pi}r(0)f(0)(2\pi)^{-2n}(\det\dot R^L(0))^2.
\end{equation}

From \eqref{s4-eXIV}, \eqref{s4-eXVI}, \eqref{s4-eXXVIII} and \eqref{s4-eXXX}, we conclude that
\begin{equation}\label{s4-eXXXI} \begin{split}
&b_{1,f}(0)\\
&=2f(0)b_1(0,0)+(2\pi)^{-n}\det\dot R^L(0)\Bigr(-\frac{1}{4\pi}\hat r(0)f(0)-
\frac{1}{4\pi}\bigr(\triangle_\omega f\bigr)(0)+\frac{1}{8\pi}r(0)f(0)\Bigr).
\end{split}
\end{equation}
Let $f=1$ in \eqref{s4-eXXXI} and notice that when $f=1$, $b_{1,f}(0)=b_1(0,0)$, we deduce
\begin{equation}\label{s4-eXXXII}
b_1(0,0)=(2\pi)^{-n}\det\dot R^L(0)\Bigr(\frac{1}{4\pi}\hat r(0)-\frac{1}{8\pi}r(0)\Bigr).
\end{equation}
Combining \eqref{s4-eXXXII} with \eqref{s4-eXXXI}, we obtain
\begin{equation}\label{s4-eXXXIII}
b_{1,f}(0)=(2\pi)^{-n}\det\dot R^L(0)\Bigr(\frac{1}{4\pi}\hat r(0)f(0)-\frac{1}{8\pi}r(0)f(0)-
\frac{1}{4\pi}\bigr(\triangle_\omega f\bigr)(0)\Bigr).
\end{equation}
From \eqref{s4-eXXXIII}, \eqref{s1-e16-II} follows. 

\subsection{The coefficient $b_{2,f}$}
To know $b_{2,f}(0)$, we have to compute
\[\begin{split}
&\triangle_0\bigr(V_\Theta f(b_0(0,z)b_1(z,0)+b_1(0,z)b_0(z,0))\bigr)(0),\\
&\triangle^2_0\bigr(\phi_1V_\Theta f(b_0(0,z)b_1(z,0)+b_1(0,z)b_0(z,0))\bigr)(0),\\
&\triangle^2_0\bigr(V_\Theta fb_0(0,z)b_0(z,0)\bigr)(0),\ \ \triangle^3_0\bigr(\phi_1V_\Theta fb_0(0,z)b_0(z,0)\bigr)(0),\\
&\triangle^4_0\bigr(\phi^2_1V_\Theta fb_0(0,z)b_0(z,0)\bigr)(0).
\end{split}\]
(See \eqref{s4-eXV}.) First, we compute $\triangle_0\bigr(V_\Theta f(b_0(0,z)b_1(z,0)+b_1(0,z)b_0(z,0))\bigr)(0)$. We need

\begin{lem}\label{s4-lII}
For $s=1,\ldots,n$, we have
\begin{equation}\label{s4-eXXXIV}\begin{split}
\rabs{\frac{\pr b_1(z,0)}{\pr z_s}}_{z=0}&=-b_1(0,0)\frac{\pr V_\Theta}{\pr z_s}(0)
+(2\pi)^{-n}\det\dot R^L(0)\Bigr(\frac{1}{4\pi}\frac{\pr\hat r}{\pr z_s}(0)-
\frac{1}{8\pi}\frac{\pr r}{\pr z_s}(0)\Bigr),\\
\rabs{\frac{\pr b_1(0,z)}{\pr\ol z_s}}_{z=0}&=-b_1(0,0)\frac{\pr V_\Theta}{\pr\ol z_s}(0)
+(2\pi)^{-n}\det\dot R^L(0)\Bigr(\frac{1}{4\pi}\frac{\pr\hat r}{\pr\ol z_s}(0)-
\frac{1}{8\pi}\frac{\pr r}{\pr\ol z_s}(0)\Bigr)
\end{split} 
\end{equation} 
and
\begin{equation}\label{s4-eXXXV}\begin{split}
&\rabs{\frac{\pr^{\abs{\alpha}+\abs{\beta}}b_1(z,0)}{\pr z^\alpha\pr\ol z^\beta}}_{z=0}=0,\qquad
\forall\alpha,\beta\in\mathbb N^n_0,\ \ \abs{\beta}>0,\\
&\rabs{\frac{\pr^{\abs{\gamma}+\abs{\delta}}b_1(0,z)}{\pr z^\gamma\pr\ol z^\delta}}_{z=0}=0,\qquad
\forall\gamma,\delta\in\mathbb N^n_0,\ \ \abs{\gamma}>0.
\end{split}
\end{equation}
\end{lem}

\begin{proof}
In view of \eqref{s3-eIII}, we see that
\[\ddbar_z b_1(z,0)=O(\abs{z}^N), \qquad \pr_z b_1(0,z)=O(\abs{z}^N), \qquad \forall N\geq 0.\]
From this observation, \eqref{s4-eXXXV} follows.

From \eqref{s4-eXXXII}, we see that
\begin{equation}\label{s4-eXXXVI}\begin{split}
b_1(z,z)&=(2\pi)^{-n}\det\dot R^L(z)\Bigr(\frac{1}{4\pi}\hat r(z)-\frac{1}{8\pi}r(z)\Bigr)\\
&=V_\omega(z)\bigr(V_\Theta(z)\bigr)^{-1}\Bigr(\frac{1}{4\pi}\hat r(z)-\frac{1}{8\pi}r(z)\Bigr).
\end{split}
\end{equation}
Note that $\frac{\pr V_\omega}{\pr z_s}(0)=\frac{\pr V_\omega}{\pr\ol z_s}(0)=0$,
$s=1,\ldots,n$, $V_\Theta(0)=1$. Thus, for $s=1,\ldots,n$,
\[\begin{split}
\rabs{\frac{\pr b_1(z,z)}{\pr z_s}}_{z=0} &=\rabs{\frac{\pr b_1(z,0)}{\pr z_s}}_{z=0} \qquad
(\text{here we used \eqref{s4-eXXXV}})\\
&=\rabs{\frac{\pr}{\pr z_s}\Bigr(V_\omega(z)\bigr(V_\Theta(z)\bigr)^{-1}\bigr(\frac{1}{4\pi}\hat r(z)-\frac{1}{8\pi}r(z)\bigr)\Bigr)}_{z=0} \\
&=-b_1(0,0)\frac{\pr V_\Theta}{\pr z_s}(0)+(2\pi)^{-n}\det\dot R^L(0)\Bigr(\frac{1}{4\pi}\frac{\pr\hat r}{\pr z_s}(0)
-\frac{1}{8\pi}\frac{\pr r}{\pr z_s}(0)\Bigr).
\end{split}\]
Similarly, for $s=1,\ldots,n$, we have
\[
\rabs{\frac{\pr b_1}{\pr\ol z_s}(0,z)}_{z=0}=-b_1(0,0)\frac{\pr V_\Theta}{\pr\ol z_s}(0)
+(2\pi)^{-n}\det\dot R^L(0)\Bigr(\frac{1}{4\pi}\frac{\pr\hat r}{\pr\ol z_s}(0)
-\frac{1}{8\pi}\frac{\pr r}{\pr\ol z_s}(0)\Bigr).
\]
\eqref{s4-eXXXIV} follows.
\end{proof}

From Lemma~\ref{s4-lI}, Lemma~\ref{s4-lII} and \eqref{s4-eXVI}, it is not difficult to calculate that
\begin{equation}\label{s4-eXXXVII}\begin{split}
&\triangle_0\bigr(V_\Theta f(b_0(0,z)b_1(z,0)+b_1(0,z)b_0(z,0))\bigr)(0)\\
&=2\sum^n_{j=1}\frac{1}{\lambda_j}\Bigr(\frac{\pr^2 V_\Theta}{\pr z_j\pr\ol z_j}(0)
-\abs{\frac{\pr V_\Theta}{\pr z_j}(0)}^2\Bigr)b_0(0,0)b_1(0,0)f\\
&\qquad +2\sum^n_{j=1}\frac{1}{\lambda_j}\frac{\pr^2f}{\pr z_j\pr\ol z_j}(0)b_0(0,0)b_1(0,0)\\
&\qquad+(2\pi)^{-2n}(\det\dot R^L(0))^2\Bigr(\frac{1}{4\pi}\sum^n_{j=1}\frac{1}{\lambda_j}
\bigr(\frac{\pr f}{\pr\ol z_j}(0)\frac{\pr\hat r}{\pr z_j}(0)+\frac{\pr f}{\pr z_j}(0)\frac{\pr\hat r}{\pr\ol z_j}(0)\bigr)\\
&\qquad-\frac{1}{8\pi}\sum^n_{j=1}\frac{1}{\lambda_j}\bigr(\frac{\pr f}{\pr\ol z_j}(0)\frac{\pr r}{\pr z_j}(0)
+\frac{\pr f}{\pr z_j}(0)\frac{\pr r}{\pr\ol z_j}(0)\bigr)\Bigr).
\end{split}
\end{equation}
From \eqref{s2-eII}, we can check that
\begin{equation}\label{s4-eXXXVIII}\begin{split}
&\frac{1}{4\pi}\sum^n_{j=1}\frac{1}{\lambda_j}\bigr(\frac{\pr f}{\pr\ol z_j}(0)\frac{\pr\hat r}{\pr z_j}(0)
+\frac{\pr f}{\pr z_j}(0)\frac{\pr\hat r}{\pr\ol z_j}(0)\bigr)
-\frac{1}{8\pi}\sum^n_{j=1}\frac{1}{\lambda_j}\bigr(\frac{\pr f}{\pr\ol z_j}(0)\frac{\pr r}{\pr z_j}(0)
+\frac{\pr f}{\pr z_j}(0)\frac{\pr r}{\pr\ol z_j}(0)\bigr)\\
&\qquad=\frac{1}{4\pi^2}\langle\,\ddbar f\,,\ddbar\hat r\,\rangle_\omega(0)
+\frac{1}{4\pi^2}\langle\,\pr f\,,\pr\hat r\,\rangle_\omega(0)
-\frac{1}{8\pi^2}\langle\,\ddbar f\,,\ddbar r\,\rangle_\omega(0)
-\frac{1}{8\pi^2}\langle\,\pr f\,,\pr r\,\rangle_\omega(0).
\end{split}
\end{equation}
From \eqref{s2-eX}, \eqref{s2-eXII-I}, \eqref{s4-eXVI}, \eqref{s4-eXXXII}, \eqref{s4-eXXXVIII} and \eqref{s4-eXXXVII}, 
we obtain 
\begin{equation}\label{s4-eXXXIX}\begin{split}
&\triangle_0\bigr(V_\Theta f(b_0(0,z)b_1(z,0)+b_1(0,z)b_0(z,0))\bigr)(0)\\
&=(2\pi)^{-2n}(\det\dot R^L(0))^2\Bigr(-\frac{1}{4\pi^2}\bigr(\hat r\bigr)^2(0)f(0)-\frac{1}{4\pi^2}\hat r(0)\bigr(\triangle_\omega f\bigr)(0)+\frac{1}{8\pi^2}r(0)\hat r(0)f(0)\\
&\quad+\frac{1}{8\pi^2}r(0)\bigr(\triangle_\omega f\bigr)(0)+\frac{1}{4\pi^2}\langle\,\ddbar f\,,\ddbar\hat r\,\rangle_\omega(0)
+\frac{1}{4\pi^2}\langle\,\pr f\,,\pr\hat r\,\rangle_\omega(0)
-\frac{1}{8\pi^2}\langle\,\ddbar f\,,\ddbar r\,\rangle_\omega(0)\\
&\quad-\frac{1}{8\pi^2}\langle\,\pr f\,,\pr r\,\rangle_\omega(0)\Bigr).
\end{split}
\end{equation} 

As \eqref{s4-eXXIX}, \eqref{s4-eXXX}, we can check that 
\begin{equation} \label{s4-eXXXIX-bis}
\begin{split}
&\triangle^2_0\bigr(\phi_1V_\Theta f(b_0(0,z)b_1(z,0)+b_1(0,z)b_0(z,0))\bigr)(0)\\
&\quad=-\frac{1}{\pi}r(0)f(0)b_0(0,0)b_1(0,0)\\
&\quad=(2\pi)^{-2n}(\det\dot R^L(0))^2\Bigr(-\frac{1}{4\pi^2}r(0)\hat r(0)f(0)+\frac{1}{8\pi^2}r^2(0)f(0)\Bigr).
\end{split}
\end{equation}

Now, we compute $\triangle^2_0\bigr(V_\Theta fb_0(0,z)b_0(z,0)\bigr)(0)$. From Lemma~\ref{s4-lI}, \eqref{s4-eXVI} and some straightforward but elementary computations, we can check that 
\begin{equation} \label{s4-eXXXXV}
\begin{split} 
&\triangle^2_0\bigr(V_\Theta fb_0(0,z)b_0(z,0)\bigr)(0)\\
&=(2\pi)^{-2n}(\det\dot R^L(0))^2f(0)\Bigr(\bigr(\triangle^2_0\log V_\Theta\bigr)(0)+\bigr(\triangle_0\log V_\Theta\bigr)^2(0)+\sum^n_{s,t=1}\frac{1}{\lambda_s\lambda_t}\abs{\frac{\pr^2\log V_\Theta}{\pr\ol z_s\pr z_t}(0)}^2\Bigr)
\\
&\quad+(2\pi)^{-2n}(\det\dot R^L(0))^2\Bigr(2\sum^n_{s=1}\frac{1}{\lambda_s}\frac{\pr f}{\pr z_s}(0)\frac{\pr}{\pr\ol z_s}\bigr(\triangle_0\log V_\Theta\bigr)(0)\\
&\quad+2\sum^n_{s=1}\frac{1}{\lambda_s}\frac{\pr f}{\pr\ol z_s}(0)\frac{\pr}{\pr z_s}\bigr(\triangle_0\log V_\Theta\bigr)(0)+2\bigr(\triangle_0f\bigr)(0)\bigr(\triangle_0\log V_\Theta\bigr)(0)\\
&\quad+2\sum^n_{s,t=1}
\frac{1}{\lambda_s\lambda_t}\frac{\pr^2f}{\pr\ol z_s\pr z_t}(0)\frac{\pr^2\log V_\Theta}{\pr z_s\pr\ol z_t}(0)+\bigr(\triangle^2_0f\bigr)(0)\Bigr).
\end{split}
\end{equation} 
Combining \eqref{s4-eXXXXV} with Theorem~\ref{s2-tI} and Corollary~\ref{s2-cI}, we obtain 
\begin{equation} \label{s4-eXXXXVII}
\begin{split}
&\triangle^2_0\bigr(V_\Theta fb_0(0,z)b_0(z,0)\bigr)(0)\\
&=(2\pi)^{-2n}(\det\dot R^L(0))^2f(0)\Bigr(\frac{1}{4\pi^2}\bigr(\triangle_\omega\hat r\bigr)(0)+\frac{1}{\pi^2}\langle\,R^{\det}_\Theta\,,{\rm Ric\,}_\omega\,\rangle_\omega(0)\\
&\quad+\frac{1}{4\pi^2}\bigr(\hat r\bigr)^2(0)+\frac{1}{\pi^2}\abs{R^{\det}_\Theta}^2_\omega(0)\Bigr)\\
&\quad+(2\pi)^{-2n}(\det\dot R^L(0))^2\Bigr(-\frac{1}{\pi^2}\langle\,\ddbar f\,,\ddbar\hat r\,\rangle_\omega(0)
-\frac{1}{\pi^2}\langle\,\pr f\,,\pr\hat r\,\rangle_\omega(0)\\
&\quad+\frac{1}{2\pi^2}\hat r(0)\bigr(\triangle_\omega f\bigr)(0)-\frac{2}{\pi^2}\langle\,\ddbar\pr f\,,R^{\det}_\Theta\,\rangle_\omega(0)\\
&\quad-\frac{1}{\pi^2}\langle\,\ddbar\pr f\,,{\rm Ric\,}_\omega\,\rangle_\omega(0)+\frac{1}{4\pi^2}\bigr(\triangle^2_\omega f\bigr)(0)\Bigr).
\end{split}
\end{equation}

Now, we compute $\triangle^3_0\bigr(\phi_1V_\Theta b_0(0,z)b_0(z,0)\bigr)(0)$. Note that $\phi_1=O(\abs{z}^4)$ and $\frac{\pr^4\phi_1}{\pr z^\alpha\pr\ol z^\beta}(0)=0$ if $\abs{\alpha}\leq1$ or $\abs{\beta}\leq1$, $\alpha, \beta\in\mathbb N^n_0$. From this observation, Lemma~\ref{s4-lI}, \eqref{s4-eXVI} and some very complicate but elementary computations, we can check that 
\begin{equation} \label{s4-eXXXXVIII}
\begin{split}
&\triangle^3_0\bigr(\phi_1V_\Theta fb_0(0,z)b_0(z,0)\bigr)(0)\\
&=(2\pi)^{-2n}(\det\dot R^L(0))^2f(0)\Bigr(\bigr(\triangle^3_0\phi\bigr)(0)+3\bigr(\triangle^2_0\phi\bigr)(0)\bigr(\triangle_0\log V_\Theta\bigr)(0)\\
&\quad+6\sum^n_{j,k,s=1}\frac{1}{\lambda_j\lambda_k\lambda_s}\frac{\pr^4\phi}{\pr\ol z_j\pr z_j\pr\ol z_k\pr z_s}(0)\frac{\pr^2\log V_\Theta}{\pr z_k\pr\ol z_s}(0)\Bigr)\\
&\quad+(2\pi)^{-2n}(\det\dot R^L(0))^2
\Bigr(3\bigr(\triangle^2_0\phi\bigr)(0)\bigr(\triangle_0f\bigr)(0)\\
&\quad+6\sum^n_{j,k,s=1}\frac{1}{\lambda_j\lambda_k\lambda_s}\frac{\pr^4\phi}{\pr\ol z_j\pr z_j\pr\ol z_k\pr z_s}(0)\frac{\pr^2f}{\pr z_k\pr\ol z_s}(0)\\
&\quad+3\sum^n_{j,k,s=1}\frac{1}{\lambda_j\lambda_k\lambda_s}\frac{\pr^5\phi}{\pr\ol z_j\pr z_j\pr\ol z_k\pr z_k\pr\ol z_s}(0)\frac{\pr f}{\pr z_s}(0)\\
&\quad+3\sum^n_{j,k,s=1}\frac{1}{\lambda_j\lambda_k\lambda_s}\frac{\pr^5\phi}{\pr\ol z_j\pr z_j\pr\ol z_k\pr z_k\pr z_s}(0)\frac{\pr f}{\pr\ol z_s}(0)\Bigr).
\end{split}
\end{equation}
Combining \eqref{s4-eXXXXVIII} with Theorem~\ref{s2-tI} and Corollary~\ref{s2-cI}, we obtain
\begin{equation} \label{s4-eXXXXX}
\begin{split}
&\triangle^3_0\bigr(\phi_1V_\Theta fb_0(0,z)b_0(z,0)\bigr)(0)\\
&=(2\pi)^{-2n}(\det\dot R^L(0))^2f(0)\Bigr(\frac{1}{4\pi^2}\bigr(\triangle_\omega r\bigr)(0)+\frac{2}{\pi^2}\abs{{\rm Ric\,}_\omega}^2_\omega(0)+\frac{1}{\pi^2}\abs{R^{TX}_\omega}^2_\omega(0)\\
&\quad+\frac{3}{4\pi^2}r(0)\hat r(0)+\frac{6}{\pi^2}\langle\,{\rm Ric\,}_\omega\,,R^{\det}_\Theta\,\rangle_\omega(0)\Bigr)\\
&\quad+(2\pi)^{-2n}(\det\dot R^L(0))^2
\Bigr(\frac{3}{4\pi^2}r(0)\bigr(\triangle_\omega f\bigr)(0)-\frac{6}{\pi^2}\langle\,\ddbar\pr f\,,{\rm Ric\,}_\omega\,\rangle_\omega(0)\\
&\quad-\frac{3}{2\pi^2}\langle\,\ddbar f\,,\ddbar r\,\rangle_\omega(0)-\frac{3}{2\pi^2}\langle\,\pr f\,,\pr r\,\rangle_\omega(0)\Bigr).
\end{split}
\end{equation} 

Now, we compute $\triangle^4_0\bigr(\phi^2_1V_\Theta fb_0(0,z)b_0(z,0)\bigr)(0)$. It is straightforward to calculate that 
\begin{equation} \label{s4-eXXXXXI}
\begin{split}
&\triangle^4_0\bigr(\phi^2_1V_\Theta fb_0(0,z)b_0(z,0)\bigr)(0)=f(0)b_0(0,0)^2\bigr(\triangle^4_0\phi_1^2\bigr)(0)\\
&=(2\pi)^{-2n}(\det\dot R^L(0))^2f(0)\Bigr(24\sum_{1\leq s,t,j,k\leq n}\frac{1}{\lambda_t\lambda_s\lambda_j\lambda_k}
\frac{\pr^4\phi}{\pr\ol z_s\pr z_s\pr\ol z_t\pr z_k}(0)\frac{\pr^4\phi}{\pr\ol z_j\pr z_j\pr z_t\pr\ol z_k}(0)\\
&\quad+6\sum_{1\leq s,t,j,k\leq n}\frac{1}{\lambda_t\lambda_s\lambda_j\lambda_k}
\abs{\frac{\pr^4\phi}{\pr\ol z_j\pr z_s\pr\ol z_t\pr z_k}(0)}^2\\
&\quad+6\sum_{1\leq s,t,j,k\leq n}\frac{1}{\lambda_t\lambda_s\lambda_j\lambda_k}
\frac{\pr^4\phi}{\pr\ol z_t\pr z_t\pr\ol z_s\pr z_s}(0)\frac{\pr^4\phi}{\pr\ol z_k\pr z_k\pr\ol z_j\pr z_j}(0)\Bigr).
\end{split}
\end{equation}
Combining \eqref{s4-eXXXXXI} with Theorem~\ref{s2-tI} and Corollary~\ref{s2-cI}, we get
\begin{equation} \label{s4-eXXXXXIII}
\begin{split}
&\triangle^4_0\bigr(\phi^2_1V_\Theta fb_0(0,z)b_0(z,0)\bigr)(0)\\
&=(2\pi)^{-2n}(\det\dot R^L(0))^2f(0)\Bigr(\frac{24}{\pi^2}\abs{{\rm Ric\,}_\omega}^2_\omega(0)+\frac{6}{\pi^2}\abs{R^{TX}_\omega}^2_\omega(0)+\frac{3}{2\pi^2}r^2(0)\Bigr).
\end{split}
\end{equation}

From \eqref{s4-eXV}, \eqref{s4-eXVI}, \eqref{s4-eXXXII}, \eqref{s4-eXXXIX}, \eqref{s4-eXXXIX-bis}, \eqref{s4-eXXXXVII}, \eqref{s4-eXXXXX}, \eqref{s4-eXXXXXIII} and some elementary computations, we get 
\begin{equation} \label{s4-eXXXXXIV} 
\begin{split} 
b_{2,f}(0)&=2f(0)b_2(0,0)+(2\pi)^{-n}f(0)\det\dot R^L(0)\Bigr(-\frac{1}{32\pi^2}\bigr(\hat r\bigr)^2(0)+\frac{1}{32\pi^2}\hat r(0)r(0)\\
&\quad-\frac{1}{128\pi^2}r^2(0)+\frac{1}{32\pi^2}\bigr(\triangle_\omega\hat r\bigr)(0)-\frac{1}{8\pi^2}\langle\,{\rm Ric\,}_\omega\,,R^{\det}_\Theta\,\rangle_\omega(0)+\frac{1}{8\pi^2}\abs{R^{\det}_\Theta}^2_\omega(0)\\
&\quad-\frac{1}{96\pi^2}\bigr(\triangle_\omega r\bigr)(0)+\frac{1}{24\pi^2}\abs{{\rm Ric\,}_\omega}^2_\omega(0)-\frac{1}{96\pi^2}\abs{R^{TX}_\omega}^2_\omega(0)\Bigr)\\
&\quad+(2\pi)^{-n}\det\dot R^L(0)\Bigr(-\frac{1}{16\pi^2}\hat r(0)\bigr(\triangle_\omega f\bigr)(0)+\frac{1}{32\pi^2}r(0)\bigr(\triangle_\omega f\bigr)(0)\\
&\quad-\frac{1}{4\pi^2}\langle\,\ddbar\pr f\,,R^{\det}_\Theta\,\rangle_\omega(0)+\frac{1}{8\pi^2}\langle\,\ddbar\pr f\,,{\rm Ric\,}_\omega\,\rangle_\omega(0)+\frac{1}{32\pi^2}\bigr(\triangle^2_\omega f\bigr)(0)\Bigr).
\end{split}
\end{equation}
Let $f=1$ in \eqref{s4-eXXXXXIV} and notice that when $f=1$, $b_{2,f}(0)=b_2(0,0)$, we deduce
\begin{equation}\label{s4-eXXXXXV} 
\begin{split}
b_2(0,0)&=-(2\pi)^{-n}\det\dot R^L(0)\Bigr(-\frac{1}{32\pi^2}\bigr(\hat r\bigr)^2(0)+\frac{1}{32\pi^2}\hat r(0)r(0)
-\frac{1}{128\pi^2}r^2(0)\\
&\quad+\frac{1}{32\pi^2}\bigr(\triangle_\omega\hat r\bigr)(0)-\frac{1}{8\pi^2}\langle\,{\rm Ric\,}_\omega\,,R^{\det}_\Theta\,\rangle_\omega(0)+\frac{1}{8\pi^2}\abs{R^{\det}_\Theta}^2_\omega(0)\\
&\quad-\frac{1}{96\pi^2}\bigr(\triangle_\omega r\bigr)(0)+\frac{1}{24\pi^2}\abs{{\rm Ric\,}_\omega}^2_\omega(0)-\frac{1}{96\pi^2}\abs{R^{TX}_\omega}^2_\omega(0)\Bigr). 
\end{split}
\end{equation}
Combining \eqref{s4-eXXXXXV} with \eqref{s4-eXXXXXIV}, we obtain
\begin{equation}\label{s4-eXXXXXVI} 
\begin{split}
b_{2,f}(0)&=(2\pi)^{-n}f(0)\det\dot R^L(0)\Bigr(\frac{1}{32\pi^2}\bigr(\hat r\bigr)^2(0)-\frac{1}{32\pi^2}\hat r(0)r(0)\\
&\quad+\frac{1}{128\pi^2}r^2(0)-\frac{1}{32\pi^2}\bigr(\triangle_\omega\hat r\bigr)(0)+\frac{1}{8\pi^2}\langle\,{\rm Ric\,}_\omega\,,R^{\det}_\Theta\,\rangle_\omega(0)-\frac{1}{8\pi^2}\abs{R^{\det}_\Theta}^2_\omega(0)\\
&\quad+\frac{1}{96\pi^2}\bigr(\triangle_\omega r\bigr)(0)-\frac{1}{24\pi^2}\abs{{\rm Ric\,}_\omega}^2_\omega(0)+\frac{1}{96\pi^2}\abs{R^{TX}_\omega}^2_\omega(0)\Bigr)\\
&\quad+(2\pi)^{-n}\det\dot R^L(0)\Bigr(-\frac{1}{16\pi^2}\hat r(0)\bigr(\triangle_\omega f\bigr)(0)+\frac{1}{32\pi^2}r(0)\bigr(\triangle_\omega f\bigr)(0)\\
&\quad-\frac{1}{4\pi^2}\langle\,\ddbar\pr f\,,R^{\det}_\Theta\,\rangle_\omega(0)+\frac{1}{8\pi^2}\langle\,\ddbar\pr f\,,{\rm Ric\,}_\omega\,\rangle_\omega(0)+\frac{1}{32\pi^2}\bigr(\triangle^2_\omega f\bigr)(0)\Bigr).
\end{split}
\end{equation}
From \eqref{s4-eXXXXXVI}, we get \eqref{s1-e16-III}. Theorem~\ref{s1-tmain} follows.

\section{The coefficients of the asymptotic expansion of the kernel of the composition of two Berezin-Toeplitz quantizations} 

Let $f, g\in C^\infty(X)$ and let $b_{j,f,g}$, $C_j(f,g)$, $j=0,1,\ldots$, be as in \eqref{e-compI} and \eqref{e-compII}
respectively.
Fix a point $p\in X$. In this section, 
we will give a method for computing $b_{j,f,g}(p)$, $C_j(f,g)$, $j=0,1,\ldots$, and we will compute the first three terms explicitly. Near $p$, we take local holomorphic coordinates $z=(z_1,\ldots,z_n)$, $z_j=x_{2j-1}+ix_{2j}$, $j=1,\ldots,n$, and local section $s$ defined in some small neighborhood $D$ of $p$ such that \eqref{s2-eI} holds. Until further notice, we work 
with this local coordinates $z$ and we identify $p$ with the point $x=z=0$. 

In view of Corollary~\ref{s3-cII}, we see that 
\begin{equation} \label{s5-eI}
(T^{(k)}_f\circ T^{(k)}_g)(0)=\int_De^{ik(\Psi(0,z)+\Psi(z,0))}b_f(0,z,k)b_g(z,0,k)f(z)V_\Theta(z)d\lambda(z)+r_k,
\end{equation}
where $d\lambda(z)=2^ndx_1dx_2\cdots dx_{2n}$, $V_\Theta$ is given by \eqref{s1-e10} and 
\[\lim_{k\to\infty}\frac{r_k}{k^N}=0,\ \ \forall N\geq0.\] 
We notice that since $b_f(z,w,k)$, $b_g(w,z,k)$ are properly supported, we have 
\begin{equation} \label{s5-eII}
b_f(0,z,k)\in C^\infty_0(D),\ \ b_g(z,0,k)\in C^\infty_0(D).
\end{equation}
We can repeat the proof of Theorem~\ref{s4-tII} and conclude that 

\begin{thm} \label{s5-tII} 
For $b_{j,f,g}$, $j=0,1,\ldots$, in \eqref{e-compI}, we have 
\begin{equation} \label{s5-eXII}
\begin{split}
&b_{j,f,g}(0)\\
&=(2\pi)^n(\det\dot{R}^L(0))^{-1}\sum_{0\leq m\leq j}\sum_{\nu-\mu=m}\sum_{2\nu\geq 4\mu}\sum_{s+t=j-m}(-1)^\mu2^{-m}
\frac{\triangle_0^\nu\bigr(\phi_1^\mu V_\Theta b_{s,f}(0,z)b_{t,g}(z,0)\bigr)(0)}{\nu!\mu!},
\end{split}
\end{equation} 
for all $j=0,1,\ldots$, where $b_{j,f}(0,z), b_{j,g}(z,0)$, $j=0,1,\ldots$, are as in \eqref{s3-eXXI}.

In particular, 
\begin{equation} \label{s5-eXIII}
b_{0,f,g}(0)=(2\pi)^n(\det\dot{R}^L(0))^{-1}b_{0,f}(0,0)b_{0,g}(0,0),
\end{equation}
\begin{equation} \label{s5-eXIV}
\begin{split}
b_{1,f,g}(0)&=(2\pi)^n(\det\dot{R}^L(0))^{-1}\Bigr(b_{0,f}(0,0)b_{1,g}(0,0)+b_{1,f}(0,0)b_{0,g}(0,0)\\
&\quad+\frac{1}{2}\triangle_0\bigr(V_\Theta b_{0,f}(0,z)b_{0,g}(z,0)\bigr)(0)
-\frac{1}{4}\triangle^2_0\bigr(\phi_1V_\Theta b_{0,f}(0,z)b_{0,g}(z,0)\bigr)(0)\Bigr)
\end{split}
\end{equation}
and 
\begin{equation} \label{s5-eXV}
\begin{split}
b_{2,f,g}(0)&=(2\pi)^n(\det\dot{R}^L(0))^{-1}\Bigr(b_{0,f}(0,0)b_{2,g}(0,0)+b_{2,f}(0,0)b_{0,g}(0,0)+b_{1,f}(0,0)b_{1,g}(0,0)\\
&\quad+\frac{1}{2}\triangle_0\bigr(V_\Theta (b_{0,f}(0,z)b_{1,g}(z,0)+b_{1,f}(0,z)b_{0,g}(z,0))\bigr)(0)\\
&\quad-\frac{1}{4}\triangle^2_0\bigr(\phi_1V_\Theta (b_{0,f}(0,z)b_{1,g}(z,0)+b_{1,f}(0,z)b_{0,g}(z,0))\bigr)(0)\\
&\quad+\frac{1}{8}\triangle^2_0\bigr(V_\Theta b_{0,f}(0,z)b_{0,g}(z,0)\bigr)(0)
-\frac{1}{24}\triangle^3_0\bigr(\phi_1V_\Theta b_{0,f}(0,z)b_{0,g}(z,0)\bigr)(0)\\
&\quad+\frac{1}{192}\triangle^4_0\bigr(\phi_1^2V_\Theta b_{0,f}(0,z)b_{0,g}(z,0)\bigr)(0)\Bigr).
\end{split}
\end{equation}
\end{thm}

\subsection{The coefficients $b_{0,f,g}$ and $C_0(f,g)$} 

From \eqref{s5-eXIII}, \eqref{s4-eXVII} and \eqref{e-compII}, we obtain 

\begin{equation} \label{s5-compI}
\begin{split}
&b_{0,f,g}(0)=(2\pi)^{-n}f(0)g(0)\det\dot{R}^L(0),\\
&C_{0,f,g}(0)=f(0)g(0).
\end{split}
\end{equation}
From \eqref{s5-compI}, we get \eqref{e-compmainI} and \eqref{e-compmainIV}.

\subsection{The coefficients $b_{1,f,g}$ and $C_1(f,g)$} 

In view of \eqref{s5-eXIV}, we see that to compute $b_{1,f,g}(0)$ we have to know which global geometric functions
at $z=0$ equal to $\triangle_0\bigr(V_\Theta b_{0,f}(0,z)b_{0,g}(z,0)\bigr)(0)$ and $\triangle_0^2\bigr(\phi_1V_\Theta b_{0,f}(0,z)b_{0,g}(z,0)\bigr)(0)$.
Now we compute $\triangle_0\bigr(V_\Theta b_{0,f}(0,z)b_{0,g}(z,0)\bigr)(0)$. First, we need

\begin{lem}\label{s5-lI}
We have
\begin{equation} \label{s5-compII}
\begin{split}
&b_{0,g}(z,0)=b_0(z,0)\sum_{\alpha\in\mathbb N^n_0,\abs{\alpha}\leq N}\frac{\pr^{\abs{\alpha}}g}{\pr z^\alpha}(0)\frac{z^\alpha}{\alpha!}+O(\abs{z}^{N+1}),\ \ \forall N\in\mathbb N_0,\\
&b_{0,f}(0,z)=b_0(0,z)\sum_{\alpha\in\mathbb N^n_0,\abs{\alpha}\leq N}\frac{\pr^{\abs{\alpha}}f}{\pr\ol z^\alpha}(0)\frac{\ol z^\alpha}{\alpha!}+O(\abs{z}^{N+1}),\ \ \forall N\in\mathbb N_0.
\end{split}
\end{equation}

\end{lem}

\begin{proof}
In view of \eqref{s3-eIII} and \eqref{s3-eXXI}, we see that 
\begin{equation} \label{s5-compIII}
\begin{split}
&\pr_zb_{0,g}(0,z)=O(\abs{z}^N),\ \ \pr_zb_0(0,z)=O(\abs{z}^N),\ \ \forall N\geq0,\\
&\ddbar_zb_{0,g}(z,0)=O(\abs{z}^N),\ \ \ddbar_zb_0(z,0)=O(\abs{z}^N),\ \ \forall N\geq0.
\end{split}
\end{equation} 
Combining \eqref{s5-compIII} with \eqref{s4-eXVII}, we obtain that $\forall \alpha, \beta\in\mathbb N^n_0$, $\abs{\beta}>0$, we have
\begin{equation}\label{s5-compIV}
\begin{split}
\rabs{\frac{\pr^{\abs{\alpha}} b_{0,g}(z,0)}{\pr z^\alpha}}_{z=0}=\rabs{\frac{\pr^{\abs{\alpha}} b_{0,g}(z,z)}{\pr z^\alpha}}_{z=0}
=\rabs{\frac{\pr^{\abs{\alpha}}(b_0(z,z)g(z))}{\pr z^\alpha}}_{z=0}=\rabs{\frac{\pr^{\abs{\alpha}}(b_0(z,0)g(z))}{\pr z^\alpha}}_{z=0}
\end{split}
\end{equation}
and 
\begin{equation}\label{s5-compIV-I}
\rabs{\frac{\pr^{\abs{\alpha}+\abs{\beta}} b_0(z,0)}{\pr z^\alpha\pr\ol z^\beta}}_{z=0}=\rabs{\frac{\pr^{\abs{\alpha}+\abs{\beta}} b_{0,g}(z,0)}{\pr z^\alpha\pr\ol z^\beta}}_{z=0}=0.
\end{equation}
From \eqref{s5-compIV} and \eqref{s5-compIV-I}, we obtain
\[b_{0,g}(z,0)=b_0(z,0)\sum_{\alpha\in\mathbb N^n_0,\abs{\alpha}\leq N}\frac{\pr^{\abs{\alpha}}g}{\pr z^\alpha}(0)\frac{z^\alpha}{\alpha!}+O(\abs{z}^{N+1}),\ \ \forall N\in\mathbb N_0.\] 
Similarly, we have
\[b_{0,f}(0,z)=b_0(0,z)\sum_{\alpha\in\mathbb N^n_0,\abs{\alpha}\leq N}\frac{\pr^{\abs{\alpha}}f}{\pr\ol z^\alpha}(0)\frac{\ol z^\alpha}{\alpha!}+O(\abs{z}^{N+1}),\ \ \forall N\in\mathbb N_0.\]
\eqref{s5-compII} follows. 
\end{proof}

Let $F(z)$ be a smooth function defined in some neighborhood of $z=0$ such that 
\begin{equation}\label{s5-compV}
F(z)=\sum_{\alpha,\beta\in\mathbb N^{n}_0,\abs{\alpha}+\abs{\beta}\leq N}
\frac{\pr^{\abs{\alpha}}f}{\pr\ol z^\alpha}(0)\frac{\pr^{\abs{\beta}}g}{\pr z^\beta}(0)\frac{\ol z^\alpha}{\alpha!}\frac{ z^\beta}{\beta!}+O(\abs{z}^N),\ \ \forall N\in\mathbb N_0.
\end{equation} 
From \eqref{s5-compII}, we see that 
\[\triangle_0\bigr(V_\Theta b_{0,f}(0,z)b_{0,g}(z,0)\bigr)(0)=\triangle_0\bigr(V_\Theta Fb_0(0,z)b_0(z,0)\bigr)(0).\]
Combining this with \eqref{s4-eXXVIII}, we obtain 
\begin{equation}\label{s5-compVI}
\triangle_0\bigr(V_\Theta b_{0,f}(0,z)b_{0,g}(z,0)\bigr)(0)=\Bigr(-\frac{1}{2\pi}\hat r(0)F(0)-\frac{1}{2\pi}\bigr(\triangle_\omega F\bigr)(0)\Bigr)
(2\pi)^{-2n}(\det\dot R^L(0))^2.
\end{equation} 
From \eqref{s5-compV}, \eqref{s2-eII} and \eqref{s2-eIV}, it is easy to check that 
\begin{equation}\label{s5-compVII}
\begin{split}
&F(0)=f(0)g(0),\\
&(\triangle_\omega F)(0)=-2\langle\,\ddbar f\,,\ddbar\ol g\,\rangle_\omega(0).
\end{split}
\end{equation}
From \eqref{s5-compVII} and \eqref{s5-compVI}, we get
\begin{equation}\label{s5-compVIII}
\begin{split}
&\triangle_0\bigr(V_\Theta b_{0,f}(0,z)b_{0,g}(z,0)\bigr)(0)\\
&=\Bigr(-\frac{1}{2\pi}\hat r(0)f(0)g(0)+\frac{1}{\pi}\langle\,\ddbar f\,,\ddbar\ol g\,\rangle_\omega(0)\Bigr)
(2\pi)^{-2n}(\det\dot R^L(0))^2.
\end{split}
\end{equation} 

As \eqref{s4-eXXIX}, \eqref{s4-eXXX}, we can check that 
\begin{equation}\label{s5-compIX}
\begin{split}
\triangle^2_0\bigr(\phi_1V_\Theta b_{0,f}(0,z)b_{0,g}(z,0)\bigr)(0)
&=\triangle^2_0\bigr(\phi_1V_\Theta Fb_0(0,z)b_0(z,0)\bigr)(0)\\
&=-\frac{1}{2\pi}r(0)f(0)g(0)(2\pi)^{-2n}(\det\dot R^L(0))^2.
\end{split}
\end{equation} 

From \eqref{s5-eXIV}, \eqref{s4-eXVII}, \eqref{s4-eXXXIII}, \eqref{s5-compVIII} and \eqref{s5-compIX}, it is straightforward 
to check that
\begin{equation}\label{s5-ecompX} \begin{split}
&b_{1,f,g}(0)\\
&=(2\pi)^{-n}\det\dot R^L(0)f(0)g(0)\Bigr(\frac{1}{4\pi}\hat r(0)-\frac{1}{8\pi}r(0)\Bigr)\\
&\quad+(2\pi)^{-n}\det\dot R^L(0)\Bigr(-\frac{1}{4\pi}\bigr(g\triangle_{\omega}f+f\triangle_{\omega}g\bigr)+\frac{1}{2\pi}\langle\,\ddbar f\,,\ddbar\ol g\,\rangle\,_\omega\Bigr)(0).
\end{split}
\end{equation}

From \eqref{s4-eXXXIII}, we know that 
\begin{equation}\label{s5-ecompXI}
\begin{split}
&b_{1,fg}(0)\\
&=(2\pi)^{-n}\det\dot R^L(0)f(0)g(0)\Bigr(\frac{1}{4\pi}\hat r(0)-\frac{1}{8\pi}r(0)\Bigr)\\
&\quad+(2\pi)^{-n}\det\dot R^L(0)\Bigr(-\frac{1}{4\pi}\triangle_\omega(fg)\Bigr)(0).
\end{split}
\end{equation}
Combining this with \eqref{s2-ecompI}, we deduce 
\begin{equation} \label{s5-ecompXII}
\begin{split}
&b_{1,fg}(0)\\
&=(2\pi)^{-n}\det\dot R^L(0)f(0)g(0)\Bigr(\frac{1}{4\pi}\hat r(0)-\frac{1}{8\pi}r(0)\Bigr)\\
&\quad+(2\pi)^{-n}\det\dot R^L(0)\Bigr(-\frac{1}{4\pi}\bigr(g\triangle_{\omega}f+f\triangle_{\omega}g\bigr)+\frac{1}{2\pi}\langle\,\ddbar f\,,\ddbar\ol g\,\rangle\,_\omega+\frac{1}{2\pi}\langle\,\pr f\,,\pr\ol g\,\rangle\,_\omega\Bigr)(0).
\end{split}
\end{equation} 
From \eqref{s5-ecompXII} and \eqref{s5-ecompX}, we conclude that 
\begin{equation}\label{s5-ecompXIII}
b_{1,f,g}(0)=b_{1,fg}(0)+(2\pi)^{-n}\det\dot R^L(0)\Bigr(-\frac{1}{2\pi}\langle\,\pr f\,,\pr\ol g\,\rangle\,_\omega\Bigr)(0).
\end{equation}
From \eqref{s5-ecompX} and \eqref{s5-ecompXIII}, \eqref{e-compmainII} follows.

Note that
\[b_{1,f,g}(0)=b_{1,fg}(0)+b_{0,C_1(f,g)}(0).\]
From this observation and \eqref{s5-ecompXIII}, \eqref{s4-eXVII}, we conclude that 
\begin{equation}\label{s5-ecompXIV}
C_1(f,g)(0)=-\frac{1}{2\pi}\langle\,\pr f\,,\pr\ol g\,\rangle\,_\omega(0).
\end{equation}
\eqref{e-compmainV} follows. 

\subsection{The coefficients $b_{2,f,g}$ and $C_2(f,g)$}

To know $b_{2,f,g}(0)$, we have to compute
\[\begin{split}
&\triangle_0\bigr(V_\Theta(b_{0,f}(0,z)b_{1,g}(z,0)+b_{1,f}(0,z)b_{0,g}(z,0))\bigr)(0),\\
&\triangle^2_0\bigr(\phi_1V_\Theta(b_{0,f}(0,z)b_{1,g}(z,0)+b_{1,f}(0,z)b_{0,g}(z,0))\bigr)(0),\\
&\triangle^2_0\bigr(V_\Theta b_{0,f}(0,z)b_{0,g}(z,0)\bigr)(0),\ \ \triangle^3_0\bigr(\phi_1V_\Theta b_{0,f}(0,z)b_{0,g}(z,0)\bigr)(0),\\
&\triangle^4_0\bigr(\phi^2_1V_\Theta b_{0,f}(0,z)b_{0,g}(z,0)\bigr)(0).
\end{split}\]
(See \eqref{s5-eXV}.) First, we compute $\triangle_0\bigr(V_\Theta (b_{0,f}(0,z)b_{1,g}(z,0)+b_{1,f}(0,z)b_{0,g}(z,0))\bigr)(0)$. We need 

\begin{lem}\label{s5-lII}
For $s=1,\ldots,n$, we have
\begin{equation}\label{s5-eXXXIV}\begin{split}
\rabs{\frac{\pr b_{1,g}(z,0)}{\pr z_s}}_{z=0}&=-b_{1,g}(0,0)\frac{\pr V_\Theta}{\pr z_s}(0)
+(2\pi)^{-n}\det\dot R^L(0)g(0)\bigr(\frac{1}{4\pi}\frac{\pr\hat r}{\pr z_s}(0)-
\frac{1}{8\pi}\frac{\pr r}{\pr z_s}(0)\bigr)\\
&+(2\pi)^{-n}\det\dot R^L(0)\frac{\pr g}{\pr z_s}(0)\bigr(\frac{1}{4\pi}\hat r(0)-\frac{1}{8\pi}r(0)\bigr)\\
&+(2\pi)^{-n}\det\dot R^L(0)\bigr(-\frac{1}{4\pi}\bigr)\frac{\pr\triangle_\omega g}{\pr z_s}(0),\\
\rabs{\frac{\pr b_{1,f}(0,z)}{\pr\ol z_s}}_{z=0}&=-b_{1,f}(0,0)\frac{\pr V_\Theta}{\pr\ol z_s}(0)
+(2\pi)^{-n}\det\dot R^L(0)f(0)\bigr(\frac{1}{4\pi}\frac{\pr\hat r}{\pr\ol z_s}(0)-
\frac{1}{8\pi}\frac{\pr r}{\pr\ol z_s}(0)\bigr)\\
&+(2\pi)^{-n}\det\dot R^L(0)\frac{\pr f}{\pr\ol z_s}(0)\bigr(\frac{1}{4\pi}\hat r(0)-\frac{1}{8\pi}r(0)\bigr)\\
&+(2\pi)^{-n}\det\dot R^L(0)\bigr(-\frac{1}{4\pi}\bigr)\frac{\pr\triangle_\omega f}{\pr\ol z_s}(0)
\end{split} 
\end{equation} 
and
\begin{equation}\label{s5-eXXXV}\begin{split}
&\rabs{\frac{\pr^{\abs{\alpha}+\abs{\beta}}b_{1,g}(z,0)}{\pr z^\alpha\pr\ol z^\beta}}_{z=0}=0,\qquad
\forall\alpha,\beta\in\mathbb N^n_0,\ \ \abs{\beta}>0,\\
&\rabs{\frac{\pr^{\abs{\gamma}+\abs{\delta}}b_{1,f}(0,z)}{\pr z^\gamma\pr\ol z^\delta}}_{z=0}=0,\qquad
\forall\gamma,\delta\in\mathbb N^n_0,\ \ \abs{\gamma}>0.
\end{split}
\end{equation}
\end{lem}

The proof of Lemma~\ref{s5-lII} is essentially the same as the proof of Lemma~\ref{s4-lII}. We omit the proof. 

From Lemma~\ref{s4-lI}, Lemma~\ref{s5-lI} and Lemma~\ref{s5-lII}, it is not difficult to calculate that
\begin{equation}\label{s5-eXXXVII}\begin{split}
&\triangle_0\bigr(V_\Theta (b_{0,f}(0,z)b_{1,g}(z,0)+b_{1,f}(0,z)b_{0,g}(z,0))\bigr)(0)\\
&=\bigr(b_{0,f}(0,0)b_{1,g}(0,0)+b_{1,f}(0,0)b_{0,g}(0,0)\bigr)\sum^n_{j=1}\frac{1}{\lambda_j}\frac{\pr^2\log V_\Theta}{\pr z_j\pr\ol z_j}(0)\\
&\quad+(2\pi)^{-n}\det\dot R^L(0)\Bigr(b_{0,g}(0,0)\sum^n_{j=1}\frac{1}{\lambda_j}\frac{\pr f}{\pr\ol z_j}(0)\bigr(\frac{1}{4\pi}\frac{\pr\hat r}{\pr z_j}(0)-\frac{1}{8\pi}\frac{\pr r}{\pr z_j}(0)\bigr)\\
&\quad+b_{0,f}(0,0)\sum^n_{j=1}\frac{1}{\lambda_j}\frac{\pr g}{\pr z_j}(0)\bigr(\frac{1}{4\pi}\frac{\pr\hat r}{\pr\ol z_j}(0)-\frac{1}{8\pi}\frac{\pr r}{\pr\ol z_j}(0)\bigr)\Bigr)\\
&\quad+(2\pi)^{-2n}(\det\dot R^L(0))^2\bigr(\frac{1}{2\pi}\hat r(0)-\frac{1}{4\pi}r(0)\bigr)\sum^n_{j=1}\frac{1}{\lambda_j}\frac{\pr f}{\pr\ol z_j}(0)\frac{\pr g}{\pr z_j}(0)\\
&\quad+(2\pi)^{-2n}(\det\dot R^L(0))^2\bigr(-\frac{1}{4\pi}\bigr)\sum^n_{j=1}\frac{1}{\lambda_j}\bigr(\frac{\pr f}{\pr\ol z_j}(0)\frac{\pr\triangle_\omega g}{\pr z_j}(0)+\frac{\pr g}{\pr z_j}(0)\frac{\pr\triangle_\omega f}{\pr\ol z_j}(0)\bigr).
\end{split}
\end{equation}
Combining \eqref{s5-eXXXVII} with Lemma~\ref{s2-lI}, Theorem~\ref{s2-tI} and Corollary~\ref{s2-cI}, we obtain 
\begin{equation}\label{s5-ecompXV}\begin{split}
&\triangle_0\bigr(V_\Theta (b_{0,f}(0,z)b_{1,g}(z,0)+b_{1,f}(0,z)b_{0,g}(z,0))\bigr)(0)\\
&=\bigr(b_{0,f}(0,0)b_{1,g}(0,0)+b_{1,f}(0,0)b_{0,g}(0,0)\bigr)\bigr(-\frac{1}{2\pi}\bigr)\hat r(0)\\
&\quad+(2\pi)^{-2n}(\det\dot R^L(0))^2\Bigr(g(0)\bigr(\frac{1}{4\pi^2}\langle\,\ddbar f\,,\ddbar\hat r\,\rangle_\omega(0)
-\frac{1}{8\pi^2}\langle\,\ddbar f\,,\ddbar r\,\rangle_\omega(0)\bigr)\\
&\quad+f(0)\bigr(\frac{1}{4\pi^2}\langle\,\pr g\,,\pr\hat r\,\rangle_\omega(0)
-\frac{1}{8\pi^2}\langle\,\pr g\,,\pr r\,\rangle_\omega(0)\bigr)+\bigr(\frac{1}{2\pi^2}\hat r(0)-\frac{1}{4\pi^2}r(0)\bigr)
\langle\,\ddbar f\,,\ddbar\ol g\,\rangle_\omega(0)\\
&\quad-\frac{1}{4\pi^2}\langle\,\ddbar f\,,\ddbar\triangle_\omega\ol g\,\rangle_\omega(0)-\frac{1}{4\pi^2}\langle\,\ddbar\triangle_\omega f\,,\ddbar\ol g\,\rangle_\omega(0)\Bigr).
\end{split}
\end{equation}

As \eqref{s4-eXXIX}, \eqref{s4-eXXX}, we can check that 
\begin{equation} \label{s5-eXXXIX-bis}
\begin{split}
&\triangle^2_0\bigr(\phi_1V_\Theta(b_{0,f}(0,z)b_{1,g}(z,0)+b_{1,f}(0,z)b_{0,g}(z,0))\bigr)(0)\\
&\quad=-\frac{1}{2\pi}r(0)\bigr(b_{0,f}(0,0)b_{1,g}(0,0)+b_{1,f}(0,0)b_{0,g}(0,0)\bigr).
\end{split}
\end{equation}

Now, we compute $\triangle^2_0\bigr(V_\Theta b_{0,f}(0,z)b_{0,g}(z,0)\bigr)(0)$. From Lemma~\ref{s5-lI}, we see that 
\begin{equation}\label{s5-ecompXVbisadd}
\triangle^2_0\bigr(V_\Theta b_{0,f}(0,z)b_{0,g}(z,0)\bigr)(0)=\triangle^2_0\bigr(V_\Theta Fb_0(0,z)b_0(z,0)\bigr)(0),
\end{equation} 
where $F$ is given by \eqref{s5-compV}. Combining \eqref{s5-ecompXVbisadd} with \eqref{s4-eXXXXVII}, we obtain
\begin{equation} \label{s5-eXXXXVII}
\begin{split}
&\triangle^2_0\bigr(V_\Theta b_{0,f}(0,z)b_{0,g}(z,0)\bigr)(0)\\
&=\triangle^2_0\bigr(V_\Theta Fb_0(0,z)b_0(z,0)\bigr)(0)\\
&=(2\pi)^{-2n}(\det\dot R^L(0))^2F(0)\Bigr(\frac{1}{4\pi^2}\bigr(\triangle_\omega\hat r\bigr)(0)+\frac{1}{\pi^2}\langle\,R^{\det}_\Theta\,,{\rm Ric\,}_\omega\,\rangle_\omega(0)\\
&\quad+\frac{1}{4\pi^2}\bigr(\hat r\bigr)^2(0)+\frac{1}{\pi^2}\abs{R^{\det}_\Theta}^2_\omega(0)\Bigr)\\
&\quad+(2\pi)^{-2n}(\det\dot R^L(0))^2\Bigr(-\frac{1}{\pi^2}\langle\,\ddbar F\,,\ddbar\hat r\,\rangle_\omega(0)
-\frac{1}{\pi^2}\langle\,\pr F\,,\pr\hat r\,\rangle_\omega(0)\\
&\quad+\frac{1}{2\pi^2}\hat r(0)\bigr(\triangle_\omega F\bigr)(0)-\frac{2}{\pi^2}\langle\,\ddbar\pr F\,,R^{\det}_\Theta\,\rangle_\omega(0)\\
&\quad-\frac{1}{\pi^2}\langle\,\ddbar\pr F\,,{\rm Ric\,}_\omega\,\rangle_\omega(0)+\frac{1}{4\pi^2}\bigr(\triangle^2_\omega F\bigr)(0)\Bigr).
\end{split}
\end{equation} 

We need 

\begin{lem} \label{s5-lIII}
We have 
\begin{equation}\label{s5-ecompXVI}
\begin{split}
&F(0)=f(0)g(0),\ \ (\ddbar F)(0)=g(0)(\ddbar f)(0),\\
&(\pr F)(0)=f(0)(\pr g)(0),\ \ (\ddbar\pr F)(0)=(\ddbar f)(0)\wedge(\pr g)(0),
\end{split}
\end{equation}
\begin{equation}\label{s5-ecompXVII}
\bigr(\triangle_\omega F\bigr)(0)=-2\langle\,\ddbar f\,,\ddbar\ol g\,\rangle_\omega(0),
\end{equation}
\begin{equation}\label{s5-ecompXVIII}
\bigr(\triangle^2_\omega F\bigr)(0)=4\langle\,D^{0,1}\ddbar f\,,D^{0,1}\ddbar\ol g\,\rangle_\omega(0)+4\langle\,\ddbar f\wedge\pr g\,,{\rm Ric\,}_\omega\,\rangle_\omega(0).
\end{equation}
\end{lem}

\begin{proof} 
\eqref{s5-ecompXVI} is easy. From \eqref{s5-compV}, \eqref{s2-eII} and \eqref{s2-eIV}, we can check that
\[\bigr(\triangle_\omega F\bigr)(0)=-2\pi\sum^n_{j=1}\frac{1}{\lambda_j}\frac{\pr f}{\pr\ol z_j}(0)\frac{\pr g}{\pr z_j}(0)=
-2\langle\,\ddbar f\,,\ddbar\ol g\,\rangle_\omega(0).\]
\eqref{s5-ecompXVII} follows. 

From \eqref{s2-eXII-bis}, \eqref{s5-compV} and \eqref{s2-ecomppreII}, we have
\begin{equation}\label{s5-ecompXIX}
\begin{split}
(\triangle_\omega^2F)(0)&=4\pi^2(\triangle^2_0F)(0)+4\langle\,\ddbar\pr F\,,{\rm Ric\,}_\omega\,\rangle_\omega(0)\\
&=4\pi^2\sum^n_{j=1}\frac{1}{\lambda_j\lambda_k}\frac{\pr^4F}{\pr z_j\pr\ol z_j\pr z_k\pr\ol z_k}(0)
+4\langle\,\ddbar f\wedge\pr g\,,{\rm Ric\,}_\omega\,\rangle_\omega(0)\\
&=4\pi^2\sum^n_{j=1}\frac{1}{\lambda_j\lambda_k}\frac{\pr^2f}{\pr\ol z_j\pr\ol z_k}(0)\frac{\pr^2g}{\pr z_j\pr z_k}(0)
+4\langle\,\ddbar f\wedge\pr g\,,{\rm Ric\,}_\omega\,\rangle_\omega(0)\\
&=4\langle\,D^{0,1}\ddbar f\,,D^{0,1}\ddbar\ol g\,\rangle_\omega(0)+
4\langle\,\ddbar f\wedge\pr g\,,{\rm Ric\,}_\omega\,\rangle_\omega(0).
\end{split}
\end{equation}
\eqref{s5-ecompXVIII} follows.
\end{proof}

Combining Lemma~\ref{s5-lIII} with \eqref{s5-eXXXXVII}, we conclude that
\begin{equation} \label{s5-ecompXX}
\begin{split}
&\triangle^2_0\bigr(V_\Theta b_{0,f}(0,z)b_{0,g}(z,0)\bigr)(0)\\
&=(2\pi)^{-2n}(\det\dot R^L(0))^2f(0)g(0)\Bigr(\frac{1}{4\pi^2}\bigr(\triangle_\omega\hat r\bigr)(0)+\frac{1}{\pi^2}\langle\,R^{\det}_\Theta\,,{\rm Ric\,}_\omega\,\rangle_\omega(0)\\
&\quad+\frac{1}{4\pi^2}\bigr(\hat r\bigr)^2(0)+\frac{1}{\pi^2}\abs{R^{\det}_\Theta}^2_\omega(0)\Bigr)\\
&\quad+(2\pi)^{-2n}(\det\dot R^L(0))^2\Bigr(-\frac{1}{\pi^2}g(0)\langle\,\ddbar f\,,\ddbar\hat r\,\rangle_\omega(0)
-\frac{1}{\pi^2}f(0)\langle\,\pr g\,,\pr\hat r\,\rangle_\omega(0)\\
&\quad-\frac{1}{\pi^2}\hat r(0)\langle\,\ddbar f\,,\ddbar\ol g\,\rangle_\omega(0)-\frac{2}{\pi^2}\langle\,\ddbar f\wedge\pr g\,,R^{\det}_\Theta\,\rangle_\omega(0)\\
&\quad+\frac{1}{\pi^2}\langle\,D^{0,1}\ddbar f\,,D^{0,1}\ddbar\ol g\,\rangle_\omega(0)\Bigr).
\end{split}
\end{equation} 

Now, we compute $\triangle^3_0\bigr(\phi_1V_\Theta b_{0,f}(0,z)b_{0,g}(z,0)\bigr)(0)$. As before, we have 
\begin{equation}\label{s5-ecompXXI}
\triangle^3_0\bigr(\phi_1V_\Theta b_{0,f}(0,z)b_{0,g}(z,0)\bigr)(0)=\triangle^3_0\bigr(\phi_1V_\Theta Fb_0(0,z)b_0(z,0)\bigr)(0).
\end{equation} 
From \eqref{s5-ecompXXI}, \eqref{s4-eXXXXX} and Lemma~\ref{s5-lIII}, it is easy to check that
\begin{equation} \label{s5-ecompXXII}
\begin{split}
&\triangle^3_0\bigr(\phi_1V_\Theta b_{0,f}(0,z)b_{0,g}(z,0)\bigr)(0)\\
&=\triangle^3_0\bigr(\phi_1V_\Theta Fb_0(0,z)b_0(z,0)\bigr)(0)\\
&=(2\pi)^{-2n}(\det\dot R^L(0))^2f(0)g(0)\Bigr(\frac{1}{4\pi^2}\bigr(\triangle_\omega r\bigr)(0)+\frac{2}{\pi^2}\abs{{\rm Ric\,}_\omega}^2_\omega(0)+\frac{1}{\pi^2}\abs{R^{TX}_\omega}^2_\omega(0)\\
&\quad+\frac{3}{4\pi^2}r(0)\hat r(0)+\frac{6}{\pi^2}\langle\,{\rm Ric\,}_\omega\,,R^{\det}_\Theta\,\rangle_\omega(0)\Bigr)\\
&\quad+(2\pi)^{-2n}(\det\dot R^L(0))^2
\Bigr(-\frac{3}{2\pi^2}r(0)\langle\,\ddbar f\,,\ddbar\ol g\,\rangle_\omega(0)-\frac{6}{\pi^2}\langle\,\ddbar f\wedge\pr g\,,{\rm Ric\,}_\omega\,\rangle_\omega(0)\\
&\quad-\frac{3}{2\pi^2}g(0)\langle\,\ddbar f\,,\ddbar r\,\rangle_\omega(0)-\frac{3}{2\pi^2}f(0)\langle\,\pr g\,,\pr r\,\rangle_\omega(0)\Bigr).
\end{split}
\end{equation} 

As \eqref{s4-eXXXXXIII}, we have 
\begin{equation} \label{s5-ecompXXIII}
\begin{split}
&\triangle^4_0\bigr(\phi^2_1V_\Theta b_{0,f}(0,z)b_{0,g}(z,0)\bigr)(0)\\
&=(2\pi)^{-2n}(\det\dot R^L(0))^2f(0)g(0)\Bigr(\frac{24}{\pi^2}\abs{{\rm Ric\,}_\omega}^2_\omega(0)+\frac{6}{\pi^2}\abs{R^{TX}_\omega}^2_\omega(0)+\frac{3}{2\pi^2}r^2(0)\Bigr).
\end{split}
\end{equation}

Combining \eqref{s5-eXV} with \eqref{s4-eXVII}, \eqref{s4-eXXXIII}, \eqref{s4-eXXXXXVI}, \eqref{s5-ecompXV}, 
\eqref{s5-eXXXIX-bis}, \eqref{s5-ecompXX}, \eqref{s5-ecompXXII}, \eqref{s5-ecompXXIII} and some very complicated but elementary computations, we deduce 
\begin{equation} \label{s5-ecompXXIV} 
\begin{split} 
b_{2,f,g}(0)&=(2\pi)^{-n}\det\dot R^L(0)f(0)g(0)\Bigr(\frac{1}{32\pi^2}\bigr(\hat r\bigr)^2(0)-\frac{1}{32\pi^2}\hat r(0)r(0)\\
&\quad+\frac{1}{128\pi^2}r^2(0)-\frac{1}{32\pi^2}\bigr(\triangle_\omega\hat r\bigr)(0)+\frac{1}{8\pi^2}\langle\,{\rm Ric\,}_\omega\,,R^{\det}_\Theta\,\rangle_\omega(0)-\frac{1}{8\pi^2}\abs{R^{\det}_\Theta}^2_\omega(0)\\
&\quad+\frac{1}{96\pi^2}\bigr(\triangle_\omega r\bigr)(0)-\frac{1}{24\pi^2}\abs{{\rm Ric\,}_\omega}^2_\omega(0)+\frac{1}{96\pi^2}\abs{R^{TX}_\omega}^2_\omega(0)\Bigr)\\
&\quad+(2\pi)^{-n}\det\dot R^L(0)\Bigr(\bigr(\frac{1}{8\pi^2}\hat r-\frac{1}{16\pi^2}r\bigr)\langle\,\ddbar f\,,\ddbar\ol g\,\rangle_\omega+\frac{1}{16\pi^2}f\bigr(\triangle_\omega g\bigr)\bigr(-\hat r+\frac{1}{2}r\bigr)\\
&\quad+\frac{1}{16\pi^2}g\bigr(\triangle_\omega f\bigr)\bigr(-\hat r+\frac{1}{2}r\bigr)-\frac{1}{4\pi^2}f\langle\,\ddbar\pr g\,,R^{\det}_\Theta\,\rangle_\omega-\frac{1}{4\pi^2}g\langle\,\ddbar\pr f\,,R^{\det}_\Theta\,\rangle_\omega\\
&\quad+\frac{1}{8\pi^2}f\langle\,\ddbar\pr g\,,{\rm Ric\,}_\omega\,\rangle_\omega
+\frac{1}{8\pi^2}g\langle\,\ddbar\pr f\,,{\rm Ric\,}_\omega\,\rangle_\omega
+\frac{1}{4\pi^2}\langle\,\ddbar f\wedge\pr g\,,{\rm Ric\,}_\omega\,\rangle_\omega\\
&\quad-\frac{1}{4\pi^2}\langle\,\ddbar f\wedge\pr g\,,R^{\det}_\Theta\,\rangle_\omega
-\frac{1}{8\pi^2}\langle\,\ddbar f\,,\ddbar\triangle_\omega\ol g\,\rangle_\omega-\frac{1}{8\pi^2}\langle\,\ddbar\triangle_\omega f\,,\ddbar\ol g\,\rangle_\omega\\
&\quad+\frac{1}{8\pi^2}\langle\,D^{0,1}\ddbar f\,,D^{0,1}\ddbar\ol g\,\rangle_\omega+\frac{1}{32\pi^2}f\bigr(\triangle^2_\omega g\bigr)+\frac{1}{32\pi^2}g\bigr(\triangle^2_\omega f\bigr)\\
&\quad+\frac{1}{16\pi^2}\bigr(\triangle_\omega g\bigr)\bigr(\triangle_\omega f\bigr)\Bigr)(0).
\end{split}
\end{equation} 

In view of \eqref{s4-eXXXXXVI}, we see that
\begin{equation}\label{s5-ecompXXV} 
\begin{split}
b_{2,fg}(0)&=(2\pi)^{-n}\det\dot R^L(0)f(0)g(0)\Bigr(\frac{1}{32\pi^2}\bigr(\hat r\bigr)^2(0)-\frac{1}{32\pi^2}\hat r(0)r(0)\\
&\quad+\frac{1}{128\pi^2}r^2(0)-\frac{1}{32\pi^2}\bigr(\triangle_\omega\hat r\bigr)(0)+\frac{1}{8\pi^2}\langle\,{\rm Ric\,}_\omega\,,R^{\det}_\Theta\,\rangle_\omega(0)-\frac{1}{8\pi^2}\abs{R^{\det}_\Theta}^2_\omega(0)\\
&\quad+\frac{1}{96\pi^2}\bigr(\triangle_\omega r\bigr)(0)-\frac{1}{24\pi^2}\abs{{\rm Ric\,}_\omega}^2_\omega(0)+\frac{1}{96\pi^2}\abs{R^{TX}_\omega}^2_\omega(0)\Bigr)\\
&\quad+(2\pi)^{-n}\det\dot R^L(0)\Bigr(-\frac{1}{16\pi^2}\hat r\bigr(\triangle_\omega(fg)\bigr)+\frac{1}{32\pi^2}r\bigr(\triangle_\omega(fg)\bigr)\\
&\quad-\frac{1}{4\pi^2}\langle\,\ddbar\pr(fg)\,,R^{\det}_\Theta\,\rangle_\omega+
\frac{1}{8\pi^2}\langle\,\ddbar\pr(fg)\,,{\rm Ric\,}_\omega\,\rangle_\omega+\frac{1}{32\pi^2}\bigr(\triangle^2_\omega(fg)\bigr)\Bigr)(0).
\end{split}
\end{equation} 
Note that $\ddbar\pr(fg)=g\ddbar\pr f+f\ddbar\pr g+\ddbar f\wedge\pr g+\ddbar g\wedge\pr f$. From this observation and 
\eqref{s2-ecompI}, \eqref{s2-ecompII}, it is straightforward to see that we can rewrite \eqref{s5-ecompXXV}: 
\begin{equation}\label{s5-ecompXXVI} 
\begin{split}
b_{2,fg}(0)&=(2\pi)^{-n}\det\dot R^L(0)f(0)g(0)\Bigr(\frac{1}{32\pi^2}\bigr(\hat r\bigr)^2(0)-\frac{1}{32\pi^2}\hat r(0)r(0)\\
&\quad+\frac{1}{128\pi^2}r^2(0)-\frac{1}{32\pi^2}\bigr(\triangle_\omega\hat r\bigr)(0)+\frac{1}{8\pi^2}\langle\,{\rm Ric\,}_\omega\,,R^{\det}_\Theta\,\rangle_\omega(0)-\frac{1}{8\pi^2}\abs{R^{\det}_\Theta}^2_\omega(0)\\
&\quad+\frac{1}{96\pi^2}\bigr(\triangle_\omega r\bigr)(0)-\frac{1}{24\pi^2}\abs{{\rm Ric\,}_\omega}^2_\omega(0)+\frac{1}{96\pi^2}\abs{R^{TX}_\omega}^2_\omega(0)\Bigr)\\
&\quad+(2\pi)^{-n}\det\dot R^L(0)\Bigr(\frac{1}{16\pi^2}g\bigr(\triangle_\omega f\bigr)\bigr(-\hat r+\frac{1}{2}r\bigr)+\frac{1}{16\pi^2}f\bigr(\triangle_\omega g\bigr)\bigr(-\hat r+\frac{1}{2}r\bigr)\\
&\quad+\bigr(\frac{1}{8\pi^2}\hat r-\frac{1}{16\pi^2}r\bigr)\langle\,\ddbar f\,,\ddbar\ol g\,\rangle_\omega+\bigr(\frac{1}{8\pi^2}\hat r-\frac{1}{16\pi^2}r\bigr)\langle\,\pr f\,,\pr\ol g\,\rangle_\omega-\frac{1}{4\pi^2}f\langle\,\ddbar\pr g\,,R^{\det}_\Theta\,\rangle_\omega\\
&\quad-\frac{1}{4\pi^2}g\langle\,\ddbar\pr f\,,R^{\det}_\Theta\,\rangle_\omega
+\frac{1}{8\pi^2}f\langle\,\ddbar\pr g\,,{\rm Ric\,}_\omega\,\rangle_\omega
+\frac{1}{8\pi^2}g\langle\,\ddbar\pr f\,,{\rm Ric\,}_\omega\,\rangle_\omega\\
&\quad+\frac{1}{4\pi^2}\langle\,\ddbar f\wedge\pr g\,,{\rm Ric\,}_\omega\,\rangle_\omega+\frac{1}{4\pi^2}\langle\,\ddbar g\wedge\pr f\,,{\rm Ric\,}_\omega\,\rangle_\omega-\frac{1}{4\pi^2}\langle\,\ddbar f\wedge\pr g\,,R^{\det}_\Theta\,\rangle_\omega\\
&\quad-\frac{1}{4\pi^2}\langle\,\ddbar g\wedge\pr f\,,R^{\det}_\Theta\,\rangle_\omega-\frac{1}{8\pi^2}\langle\,\ddbar f\,,\ddbar\triangle_\omega\ol g\,\rangle_\omega-\frac{1}{8\pi^2}\langle\,\ddbar\triangle_\omega f\,,\ddbar\ol g\,\rangle_\omega\\
&\quad-\frac{1}{8\pi^2}\langle\,\pr f\,,\pr \triangle_\omega\ol g\,\rangle_\omega-\frac{1}{8\pi^2}\langle\,\pr\triangle_\omega f\,,\pr\ol g\,\rangle_\omega+\frac{1}{8\pi^2}\langle\,D^{0,1}\ddbar f\,,D^{0,1}\ddbar\ol g\,\rangle_\omega\\
&\quad+\frac{1}{8\pi^2}\langle\,D^{1,0}\pr f\,,D^{1,0}\pr\ol g\,\rangle_\omega
+\frac{1}{4\pi^2}\langle\,\ddbar\pr f\,,\ddbar\pr\ol g\,\rangle_\omega
+\frac{1}{32\pi^2}f\bigr(\triangle^2_\omega g\bigr)+\frac{1}{32\pi^2}g\bigr(\triangle^2_\omega f\bigr)\\
&\quad+\frac{1}{16\pi^2}\bigr(\triangle_\omega g\bigr)\bigr(\triangle_\omega f\bigr)\Bigr)(0).
\end{split}
\end{equation} 

Combining \eqref{s5-ecompXXVI} with \eqref{s5-ecompXXIV}, we obtain
\begin{equation} \label{s5-ecompXXVII}
\begin{split}
b_{2,f,g}(0)&=b_{2,fg}(0)+(2\pi)^{-n}\det\dot{R}^L(0)\Bigr(-\frac{1}{4\pi^2}\langle\,\ddbar g\wedge\pr f\,,{\rm Ric\,}_\omega\,\rangle_\omega+\frac{1}{4\pi^2}\langle\,\ddbar g\wedge\pr f\,,R^{\det}_\Theta\,\rangle_\omega \\
&\quad+\frac{1}{8\pi^2}\langle\,\pr\triangle_\omega f\,,\pr\ol g\,\rangle_\omega
+\frac{1}{8\pi^2}\langle\,\ddbar\triangle_\omega g\,,\ddbar\,\ol f\,\rangle_\omega
-\frac{1}{8\pi^2}\langle\,D^{1,0}\pr f\,,D^{1,0}\pr\ol g\,\rangle_\omega\\
&\quad-\frac{1}{4\pi^2}\langle\,\ddbar\pr f\,,\ddbar\pr\ol g\,\rangle_\omega+\frac{1}{8\pi^2}\langle\,\pr f\,,\pr\ol g\,\rangle_\omega(-\hat r+\frac{1}{2}r)\Bigr)(0).
\end{split}
\end{equation}
From \eqref{s5-ecompXXVII}, \eqref{e-compmainIII} follows.

Now, we compute $C_2(f,g)(0)$. Note that
\begin{equation}\label{s5-ecompXXVIII}
b_{2,f,g}(0)=b_{2,fg}(0)+b_{1,C_1(f,g)}(0)+b_{0,C_2(f,g)}(0).
\end{equation} 

In view of \eqref{s5-ecompXIV} and \eqref{s4-eXXXIII}, we know that
\begin{equation}\label{s5-ecompXXIX}
b_{1,C_1(f,g)}(0)=(2\pi)^{-n}\det\dot R^L(0)\Bigr(\bigr(-\frac{1}{8\pi^2}\hat r+\frac{1}{16\pi^2}r\bigr)\langle\,\pr f\,,\pr\ol g\,\rangle_\omega+\frac{1}{8\pi^2}\triangle_\omega(\langle\,\pr f\,,\pr\ol g\,\rangle_\omega\bigr)\Bigr)(0).
\end{equation} 

We need 

\begin{lem} \label{s5-lIV}
We have 
\begin{equation}\label{s5-ecompXXX}
\begin{split}
\triangle_\omega(\langle\,\pr f\,,\pr\ol g\,\rangle_\omega\bigr)(0)&=-2\langle\,\ddbar g\wedge\pr f\,,{\rm Ric\,}_\omega\,\rangle(0)+\langle\,\pr\triangle_\omega f\,,\pr\ol g\,\rangle_\omega(0)+
\langle\,\ddbar\triangle_\omega g\,,\ddbar\,\ol f\,\rangle_\omega(0)\\
&\quad-2\langle\,\ddbar\pr f\,,\ddbar\pr\ol g\,\rangle_\omega(0)
-2\langle\,D^{1,0}\pr f\,,D^{1,0}\pr\ol g\,\rangle_\omega(0).
\end{split}
\end{equation}
\end{lem}

\begin{proof}
From \eqref{s2-eII}, it is straightforward to see that 
\begin{equation}\label{s5-ecompXXXI}
\langle\,\pr f\,,\pr\ol g\,\rangle_\omega=\pi\sum^n_{j=1}\frac{1}{\lambda_j}\frac{\pr f}{\pr z_j}\frac{\pr g}{\pr\ol z_j}
-\pi\sum^n_{j,k=1}\frac{1}{\lambda_j\lambda_k}\frac{\pr^2\phi_1}{\pr\ol z_j\pr z_k}\frac{\pr f}{\pr z_j}\frac{\pr g}{\pr\ol z_k}+O(\abs{z}^3).
\end{equation} 
From \eqref{s5-ecompXXXI} and \eqref{s2-eIV} and notice that $\phi_1=O(\abs{z}^4)$, we can check that
\begin{equation}\label{s5-ecompXXXII}
\begin{split}
\triangle_\omega(\langle\,\pr f\,,\pr\ol g\,\rangle_\omega\bigr)(0)&=-2\pi^2\sum^n_{j,s=1}\frac{1}{\lambda_j\lambda_s}\frac{\pr^3f}{\pr z_j\pr z_s\pr\ol z_s}(0)\frac{\pr g}{\pr\ol z_j}(0)\\
&\quad-
2\pi^2\sum^n_{j,s=1}\frac{1}{\lambda_j\lambda_s}\frac{\pr^3g}{\pr\ol z_j\pr z_s\pr\ol z_s}(0)\frac{\pr f}{\pr z_j}(0)\\
&\quad-2\pi^2\sum^n_{j,s=1}\frac{1}{\lambda_j\lambda_s}\frac{\pr^2f}{\pr z_j\pr z_s}(0)\frac{\pr^2g}{\pr\ol z_j\pr\ol z_s}(0)\\
&\quad-2\pi^2\sum^n_{j,s=1}\frac{1}{\lambda_j\lambda_s}\frac{\pr^2f}{\pr z_j\pr\ol z_s}(0)\frac{\pr^2g}{\pr\ol z_j\pr z_s}(0)\\
&\quad+2\pi^2\sum^n_{j,k,s=1}\frac{1}{\lambda_j\lambda_k\lambda_s}\frac{\pr^4\phi}{\pr\ol z_j\pr z_k\pr z_s\pr\ol z_s}(0)\frac{\pr f}{\pr z_j}(0)\frac{\pr g}{\pr\ol z_k}(0).
\end{split}
\end{equation}
Combining \eqref{s5-ecompXXXII} with \eqref{s2-ecomppreIII}, \eqref{s2-ecomppreIV}, \eqref{s2-ecomppreI}, 
\eqref{s2-ecomppre0} and \eqref{s2-ecomppreb0}, we obtain
\[
\begin{split}
\triangle_\omega(\langle\,\pr f\,,\pr\ol g\,\rangle_\omega\bigr)(0)&=-2\langle\,\ddbar g\wedge\pr f\,,{\rm Ric\,}_\omega\,\rangle(0)+\langle\,\pr\triangle_\omega f\,,\pr\ol g\,\rangle_\omega(0)+
\langle\,\ddbar\triangle_\omega g\,,\ddbar\,\ol f\,\rangle_\omega(0)\\
&\quad-2\langle\,\ddbar\pr f\,,\ddbar\pr\ol g\,\rangle_\omega(0)
-2\langle\,D^{1,0}\pr f\,,D^{1,0}\pr\ol g\,\rangle_\omega(0).
\end{split}\]
\eqref{s5-ecompXXX} follows.
\end{proof} 

From \eqref{s5-ecompXXX} and \eqref{s5-ecompXXIX}, we get 
\begin{equation}\label{s5-ecompXXXIII}
\begin{split}
b_{1,C_1(f,g)}(0)&=(2\pi)^{-n}\det\dot R^L(0)\Bigr(\bigr(-\frac{1}{8\pi^2}\hat r+\frac{1}{16\pi^2}r\bigr)\langle\,\pr f\,,\pr\ol g\,\rangle_\omega-\frac{1}{4\pi^2}\langle\,\ddbar g\wedge\pr f\,,{\rm Ric\,}_\omega\,\rangle\\
&\quad+
\frac{1}{8\pi^2}\langle\,\pr\triangle_\omega f\,,\pr\ol g\,\rangle_\omega+\frac{1}{8\pi^2}
\langle\,\ddbar\triangle_\omega g\,,\ddbar\,\ol f\,\rangle_\omega
-\frac{1}{4\pi^2}\langle\,\ddbar\pr f\,,\ddbar\pr\ol g\,\rangle_\omega\\
&\quad-\frac{1}{4\pi^2}\langle\,D^{1,0}\pr f\,,D^{1,0}\pr\ol g\,\rangle_\omega\Bigr)(0).
\end{split}
\end{equation} 
Combining \eqref{s5-ecompXXXIII} with \eqref{s5-ecompXXVII}, we obtain 
\begin{equation} \label{s5-ecompXXXIII-I}
\begin{split}
b_{2,f,g}(0)&=b_{2,fg}(0)+b_{1,C_1(f,g)}(0)\\
&\quad+(2\pi)^{-n}\det\dot{R}^L(0)\Bigr(
\frac{1}{8\pi^2}\langle\,D^{1,0}\pr f\,,D^{1,0}\pr\ol g\,\rangle_\omega+\frac{1}{4\pi^2}\langle\,\ddbar g\wedge\pr f\,,R^{\det}_\Theta\,\rangle_\omega\Bigr)(0).
\end{split}
\end{equation} 
Combining \eqref{s5-ecompXXXIII-I} with \eqref{s5-ecompXXVIII}, \eqref{s4-eXVII}, we conclude that 
\[C_2(f,g)(0)=\frac{1}{8\pi^2}\langle\,D^{1,0}\pr f\,,D^{1,0}\pr\ol g\,\rangle_\omega(0)+\frac{1}{4\pi^2}\langle\,\ddbar g\wedge\pr f\,,R^{\det}_\Theta\,\rangle_\omega(0).\]
We obtain \eqref{e-compmainVI} and Theorem~\ref{s1-tmaincomp} follows. 

\smallskip

\noindent
{\small\emph{
\textbf{Acknowledgements.} The author would like to thank the institute of Mathematics of Academia Sinica, Taiwan, for
offering excellent working conditions during the month of July, 2011. Furthermore, the author is grateful to Prof. George Marinescu and Prof. Xiannon Ma for 
comments and useful suggestions on an early draft of the manuscript.}}

\end{document}